\newif\ifarxiv
\arxivtrue
\ifarxiv
  \documentclass[a4paper,twoside]{amsart}
\else
  \documentclass{lms}
\fi

\usepackage[utf8]{inputenc}
\usepackage{amsmath}
\usepackage{amssymb}
\usepackage{tikz}
\usepackage{mathrsfs}
\usepackage{hyperref}
\usetikzlibrary{matrix,arrows,calc}

\title[Motivic classes of some classifying stacks]
{Motivic classes of some classifying stacks}
\author{Daniel Bergh}
\ifarxiv
\else
  \classno{14A20 (primary), 20G99, 14L30 (secondary)}
\fi

\ifarxiv
\address{
Daniel Bergh
Department of Mathematics,
Stockholm University,
SE-106 91 Stockholm, Sweden}
\email{dbergh@gmail.com}
\fi

\date{}

\ifarxiv
\theoremstyle{plain}
\fi
\newtheorem{maintheorem}{Theorem}

\newtheorem{theorem}{Theorem}
\newtheorem{prop}[theorem]{Proposition}
\newtheorem{lemma}[theorem]{Lemma}
\newtheorem{corollary}[theorem]{Corollary}

\ifarxiv
\theoremstyle{definition}
\def\newnumbered{\newtheorem}
\fi
\newnumbered{definition}{Definition}

\ifarxiv
\theoremstyle{remark}
\def\newunnumbered{\newtheorem*}
\fi

\ifarxiv
\def\thmauth#1{#1}
\else
\def\thmauth#1{(#1)}
\fi

\newunnumbered{note}{Note}
\newunnumbered{remark}{Remark}

\numberwithin{theorem}{section}
\numberwithin{equation}{section}
\numberwithin{definition}{section}

\DeclareMathOperator{\Spec}{Spec}

\DeclareMathOperator{\Lie}{Lie}
\DeclareMathOperator{\Pic}{Pic}
\DeclareMathOperator{\K0}{K_0}

\DeclareMathOperator{\Hom}{Hom}
\DeclareMathOperator{\Isom}{Isom}

\DeclareMathOperator{\GL}{GL}
\DeclareMathOperator{\PGL}{PGL}
\DeclareMathOperator{\Mon}{N}
\DeclareMathOperator{\PMon}{PN}
\DeclareMathOperator{\SL}{SL}

\DeclareMathOperator{\Sp}{Sp}
\DeclareMathOperator{\diag}{diag}
\DeclareMathOperator{\id}{id}
\DeclareMathOperator{\HH}{H}
\DeclareMathOperator{\Ext}{Ext}
\DeclareMathOperator{\Coh}{Coh}

\DeclareMathOperator{\Symm}{Symm}
\DeclareMathOperator{\Conf}{Conf}
\newcommand{\Res}[0]{\ensuremath{R}}

\newcommand{\sPic}[0]{\ensuremath{\underline{\Pic}}}
\newcommand{\sHom}[0]{\ensuremath{\underline{\Hom}}}
\newcommand{\sIsom}[0]{\ensuremath{\underline{\Isom}}}

\newcommand{\stExt}[1]{\ensuremath{\stack{E}\mathrm{xt}^#1}}

\newcommand{\cat}[1]{\ensuremath{\mathrm{#1}}}
\newcommand{\stack}[1]{\ensuremath{\mathcal{#1}}}

\newcommand{\sheaf}[1]{\ensuremath{\mathscr{#1}}}
\newcommand{\contr}[1]{\ensuremath{\times^{#1}}}

\newcommand{\catset}[1]{\cat{Set}}
\newcommand{\cd}[1]{\ensuremath{{#1^\vee}}}
\newcommand{\term}[1]{{\em #1}}

\newcommand{\ec}[0]{\ensuremath{\chi_\mathrm{c}}}

\newcommand{\pMC}[2]{\stack{M}_{#1, (#2)}}
\newcommand{\MC}[2]{\stack{M}_{#1, #2}}

\newcommand{\UC}[0]{\stack{E}}
\newcommand{\Hns}[0]{\ensuremath{H_\mathrm{ns}}}
\newcommand{\Hsing}[0]{\ensuremath{H_\mathrm{sing}}}

\newcommand{\ZZ}[0]{\ensuremath{\mathbb{Z}}}

\newcommand{\QQ}[0]{\ensuremath{\mathbb{Q}}}
\newcommand{\CC}[0]{\ensuremath{\mathbb{C}}}
\newcommand{\GF}[1]{\ensuremath{\mathbb{F}_{#1}}}

\newcommand{\LL}[0]{\ensuremath{\mathbb{L}}}
\newcommand{\PP}[0]{\ensuremath{\mathbb{P}}}
\newcommand{\GG}[0]{\ensuremath{\mathbb{G}}}
\newcommand{\GGm}[0]{\ensuremath{\mathbb{G}_\mathrm{m}}}
\newcommand{\GGa}[0]{\ensuremath{\mathbb{G}_\mathrm{a}}}
\renewcommand{\SS}[0]{\ensuremath{\mathrm{S}}}
\renewcommand{\AA}[0]{\ensuremath{\mathbb{A}}}
\newcommand{\BB}[0]{\ensuremath{\mathrm{B}}}

\newcommand{\KK}[0]{\ensuremath{\mathrm{K}}}

\newcommand{\Bu}[1]{\ensuremath{\mathrm{A}(#1)}}
\newcommand{\SCF}[1]{\ensuremath{\mathrm{SCF}(#1)}}

\newcommand{\lcat}[1]{\ensuremath{\mathrm{#1}}}
\newcommand{\cAS}[0]{\ensuremath{\lcat{Space}}}
\newcommand{\cVar}[0]{\ensuremath{\lcat{Var}}}
\newcommand{\cSch}[0]{\ensuremath{\lcat{Sch}}}
\newcommand{\cStack}[0]{\ensuremath{\lcat{Stack}}}

\newcommand{\axiom}[1]{\textbf{#1}}

\newcommand{\gsi}[0]{\axiom{GS1}}
\newcommand{\gsii}[0]{\axiom{GS2}}

\newcommand{\todo}[1]{
}

\def\hcolon{\mathchoice{\!:\!}{\!:\!}{:}{:}}

\begin{document}
\maketitle
\begin{abstract}
We prove that the class of the classifying stack $\BB \PGL_n$ is the
multiplicative inverse of the class of the projective linear group $\PGL_n$
in the Grothendieck ring of stacks $\KK_0(\cStack_k)$ for $n = 2$ and $n = 3$
under mild conditions on the base field $k$.
In particular, although it is known that the multiplicativity relation
$\{T\} = \{S\}\cdot\{\PGL_n\}$ does not hold for all $\PGL_n$-torsors $T \to S$,
it holds for the universal $\PGL_n$-torsors for said~$n$.
\end{abstract}
\section*{Introduction}
Recall that the Grothendieck ring of varieties over a field $k$ is defined as
the free group on isomorphism classes $\{X\}$ of varieties $X$ subject to the
\term{scissors relations} $\{X\} = \{X\setminus Z\} + \{Z\}$ for closed
subvarieties $Z \subset X$. We denote this ring by $\K0(\cVar_k)$. Its
multiplicative structure is induced by $\{X\}\cdot\{Y\} = \{X\times Y\}$ for
varieties $X$ and~$Y$.

A fundamental question regarding $\K0(\cVar_k)$ is for which groups $G$ the
\term{multiplicativity relation} $\{T\} = \{G\}\cdot\{S\}$ holds for all
$G$-torsors $T\to S$. If $G$ is \term{special}, in the sense of Serre and
Grothendieck, this relation always holds. However, it was shown by Ekedahl that
if $G$ is a non-special connected linear group, then there exists a $G$-torsor
$T \to S$ for which $\{T\} \neq \{G\}\cdot\{S\}$ in the ring $\K0(\cVar_\CC)$
and its completion $\widehat{\KK}_0(\cVar_\CC)$~\cite[Theorem~2.2]{ekedahl2009ap}.

One can also construct a Grothendieck ring $\K0(\cStack_k)$ for the larger
2-category of algebraic stacks. We will recall its precise definition in
Section~\ref{gro}. This ring was studied by Ekedahl in a series of
preprint~\cite{ekedahl2009ap,ekedahl2009gg,ekedahl2009fg}.
Similar constructions have also been studied independently by other
authors~\cite{behrend2007,joyce2007,toen2005}.

The ring $\K0(\cStack_k)$ is the localisation of the ring $\K0(\cVar_k)$ in
the class $\LL$ of the affine line and in all cyclotomic
polynomials in $\LL$. In particular, the canonical map from the
Grothendieck of varieties to its completion factors through $\K0(\cStack_k)$, so
the results regarding the multiplicativity relations for torsors stated above
hold also in $\K0(\cStack_k)$. When considering stacks, the \term{universal}
torsor $\Spec k \to \BB G$ for a given group $G$ is of particular interest.
Here $\BB G$ denotes the classifying stack for the group $G$. This torsor
is universal in the sense that every other torsor may be obtained as a base
change of it. Since the class of the base $\Spec k$ is the multiplicative
identity in $\K0(\cStack_k)$, the multiplicativity relation for the
universal torsor implies that the class of $\BB G$ is the inverse of the
class of the group $G$ itself. In this article, we study the multiplicativity
relation for the universal $\PGL_n$-torsors for some small $n$. We state the main
result.
\label{int-result}
\begin{maintheorem}
\label{mainBPGL}
Let $n = 2$ or $n = 3$, and let $k$ be any field in which $n!$ is invertible,
and which contains all $n$-th roots of unity. Then we have
$\{\BB \PGL_n\} = \{\PGL_n\}^{-1}$ in $\KK_0(\cStack_k)$.
\end{maintheorem}
\noindent
Note that the groups $\PGL_n$ are non-special for $n > 1$. In particular, it
follows that the multiplicativity relation for a universal $G$-torsor does
{\em not} imply the multiplicativity relation for arbitrary $G$-torsors. This
gives a negative answer to the question raised in \cite[Remark~3.3]{behrend2007}.
Other counter-examples to this is given in the recent paper \cite{dhillon2014},
where the motivic classes of the classifying stack for orthogonal linear groups are
computed.

The techniques for computing the classes involves finding suitable linear
representations for the groups and then stratifying the spaces corresponding to
these representations into orbits in a similar way as demonstrated in
\cite[Section~4]{ekedahl2009gg}. In the case $n = 3$, we use the natural action
of $\PGL_3$ on the space of plane curves of degree~3. The computations make
use of the simple classification of such curves. It should be emphasised
that this method does not scale well with $n$. Presumably, already the
case $n = 4$ would be well out of reach.

In the process of computing the class of the classifying stacks $\BB \PGL_n$,
we will also compute the corresponding classes for several other groups. Most
notably, the subgroup $\PMon_n$ of monomial matrices in $\PGL_n$ shows up, both in the
case $n = 2$ and $n = 3$. Since this group is not connected, one should not
expect the class of $\BB \PMon_n$ to be the inverse of the class of $\PMon_n$. Instead
the expected class is $\LL^{n(n-1)}/\{\PGL_n\}$. We prove that for small $n$ this
is the actual class.
\begin{maintheorem}
\label{mainBPN}
Let $n = 2$ or $n = 3$, and let $\PMon_n$ denote the subgroup of monomial matrices
in $\PGL_n$. Then we have the relation
$\{\BB \PMon_n\} = \LL^{n(n-1)}/\{\PGL_n\}$
in $\KK_0(\cStack_k)$ for any field~$k$.
\end{maintheorem}
\noindent
By {\em expected} in this context, we mean that given that $\BB \PMon_n$ is a
rational function in $\LL$, it has to be this class. To prove that the
class actually is rational in $\LL$, we use the fact that $\BB \PMon_n$ can be
identified with the classifying stack for a torus relative to the base $\BB
\Sigma_n$. Here $\Sigma_n$ denotes the symmetric group of $n$ symbols. For small
$n$ the dual of this torus is \term{stably rational} in a relative sense
which we make precise in Section~\ref{tori}. But this is exceptional for small
$n$ and is known to be false for large $n$. Hence the method of proof does not
generalise. In fact, the failure of the torus of being stably rational might
even be taken as an indication that the class of $\BB \PMon_n$ might not be the
expected for large $n$.

\subsection*{Outline.}
In Section~\ref{prel}, some preliminaries are given in
order to give references and establish notation and conventions. Section~\ref{gro}
summaries some basic results from \cite{ekedahl2009gg} on the Grothendieck ring
of stacks. In Section~\ref{tori}, we study certain universal tori and compute
their classes as well as the classes of their classifying stacks in
$\K0(\cStack_k)$. This generalises and applies results from a non-stacky context
by Rökaeus \cite{rokeaus2007}. In the following section, we use this to get a
proof of Theorem~\ref{mainBPN}.

The rest of the sections are dedicated to the proof of Theorem~\ref{mainBPGL}.
We relate the problem to computing the class of the moduli space of plane degree 3 curves
up to projective equivalence in Section~\ref{com}, as well as computing
the class of the subspace of singular curves. Here we use some information
regarding stabilisers of plane degree~3 curves, which are summarised in
Appendix~\ref{sin}.

Finally, we compute the class of the open locus of smooth curves.
This is done by describing the locus as the classifying stack of the
3-torsion subgroup of the universal curve over the moduli stack of elliptic
curves. This moduli description is established in Section~\ref{mod}, and the
actual computation of the class is done in Section~\ref{com-smo}.

\subsection*{Acknowledgements}
The problem studied in this article was suggested to me by Torsten
Ekedahl (1955--2011). I am grateful for the privilege of having
been his student. Furthermore, I would like to thank my advisor, David Rydh, for
his support and for his many suggestions on improvements to this article.
Finally, I would like to thank Jochen Heinloth for his valuable comments.

\section{Preliminaries}
\label{prel}
By default, our schemes, algebraic spaces and algebraic stacks will be of
finite presentation over a field.
We will also assume that our stacks have affine, but not necessarily finite,
stabilisers.
In particular, this ensures that our algebraic spaces have open dense
schematic loci~\cite[Proposition~II.6.7]{knutson1973},
and that they admit finite stratifications into locally
closed subspaces which are schemes.
Furthermore, the conditions on algebraic stacks imply that they
admit finite stratifications into locally closed substacks which are global quotients
of schemes by general linear groups~\cite[Proposition~3.5.9]{kresch1999}.

We will often use the \term{functor of points} perspective, and will frequently
identify schemes and varieties with their functor of points. When talking
about sheaves, we will, by default, work over the site of schemes over a given
base, and we will use the fppf topology. The symbols $\emptyset$ and $\ast$ are
used for the initial and terminal objects in the associated sheaf category. We
will usually refer to sheaves of groups simply as groups.

Torsors will play an important role in this work. The canonical reference for
torsors is Giraud's book on non-abelian cohomology \cite{giraud1971}. A shorter
and more basic introduction is given in Milne's book on étale cohomology
\cite[Section III.4]{milne1980}. Recall that a \term{pseudo-torsor} for a group
$G$ is a sheaf of $G$-sets $T$ on
which $G$ acts freely and transitively. If $G$ is non-commutative, we will
usually assume that $G$ acts on $T$ from the right. By free and transitive, we
mean that the map $T \times G \to T \times T$ given by
$(t, g) \mapsto (t, t\cdot g)$ is an isomorphism. If in addition the
map $T \to \ast$ is a local epimorphism, we call $T$ a \term{torsor}. The
torsors for a group $G$ are classified by the classifying stack $\BB G$. When we
want to be explicit about the base $\stack{S}$, which in general will be a stack,
we use the notation $\BB_{\stack{S}} G$.

Given a $G$-torsor $T$ and a left $G$-space $F$, we may form the
\term{contraction product} $T \contr{G} F$. This is defined as the
quotient of the product $T \times F$ by the equivalence
relation $(t\cdot g, x) \sim (t, g\cdot x)$ for any $g \in G$. The contraction
product is locally isomorphic to $F$, and we call it the $F$-\term{fibration}
associated to the torsor $T$. The torsor may be recovered by taking
$\sIsom_G(T \contr{G} F, F)$, giving an equivalence of categories between
$F$-fibrations with structure group $G$ and $G$-torsors.

We recall some terminology and basic facts about groups of multiplicative
type, the standard reference being \cite[Exposé VIII--X]{SGA3III}. A more
elementary treatment in the case where the base is affine is given
in \cite{waterhouse1979}.

Given a group $G$, we may consider its \term{Cartier dual}, which 
we denote by $\cd{G}$. This is defined as the sheaf $\sHom_{gr}(G, \GGm)$.
By a \term{diagonalisable} group, we mean a finite product of groups
$\mu_n$ of $n$-th roots of unity and multiplicative groups $\GGm$.
These are exactly the Cartier duals of finitely generated abelian groups.
A group which is diagonalisable after a finite étale base change is called a
group of \term{multiplicative type}. If no base change is needed, the group is
said to be \term{split}. Note that these definitions are less general than those
given in \cite{SGA3III}. In particular, all our groups of multiplicative type are assumed
to be isotrivial and of finite type.

A group $T$ over a base $S$ which locally is the Cartier dual of
a lattice $L$ is called a \term{torus}. The torus $T$ is determined by
the lattice $L$ together with a continuous action of the étale fundamental group
of $S$. The torus corresponding to the dual lattice $\Hom_\ZZ(L, \ZZ)$ is called
the dual torus, and we denote it by $T^\circ$. A torus is called \term{quasi-split}
if it corresponds to a permutation representation of the étale fundamental group of
the base. Such a group can also be described as the Weil restriction $\Res_{S'/S}\GGm$
of $\GGm$ along a finite étale morphism $S'\to S$.

Finally, we recall some facts about \term{special} groups. These were studied
by Serre and Grothendieck \cite{chevalley1958}, and they made a complete
classification in the case of reduced groups over an algebraically closed field.
We make a definition, which also works when we are working with group
objects over stacks.

\begin{definition}
Let $\stack{S}$ be an algebraic stack, and let $G$ be an algebraic group over
$\stack{S}$. We say that $G$ is \term{special} provided that every $G$-torsor
over any field $K$ over $\stack{S}$ is trivial. 
\end{definition}

\noindent
Examples of special groups are $\GL_n$, $\SL_n$, $\Sp_{2n}$, $\GGa$ and $\GGm$ and
extensions thereof, whereas $\PGL_n$ is not special. In particular, split tori
are special since they are products of $\GG_m$. Non-split tori need not be special in
general, but quasi-split tori are, also when considered over a general base.

\begin{prop}
A quasi-split torus $T$ over any base stack $\stack{S}$ is special.
\end{prop}
\begin{proof}
We may assume that $\stack{S} = S$ is the spectrum of a field.
First note that $T$ is isomorphic to the group of units of a
locally free sheaf of $\sheaf{O}_S$-algebras $\sheaf{A}$ of finite rank.
Indeed, let $S' \to S$ be a Galois extension splitting $T$ with corresponding Galois group $\Gamma$.
Then, since $T$ is quasi-trivial, it corresponds to the $\Gamma$-equivariant
sheaf of groups $\sheaf{O}^\times_{S'}\times \cdots \times \sheaf{O}^\times_{S'}$
where $\Gamma$ acts by permuting the factors. This is on the other hand the
group of units in the $\Gamma$-equivariant sheaf of $\sheaf{O}_{S'}$-algebras
$\sheaf{O}_{S'}\times \cdots \times \sheaf{O}_{S'}$ with the corresponding
permutation action. This sheaf of algebras descends to our desired $\sheaf{A}$
on $S$.

By flat descent for quasi-coherent sheaves, the fibred category of sheaves
of $\sheaf{A}$-modules is a stack for the fppf topology. The
$\sheaf{A}^\times$-torsors classify the rank 1 locally free sheaves for $\sheaf{A}$.
Such sheaves trivialise Zariski-locally on $\Spec\sheaf{A}$, and since $\Spec\sheaf{A}$ is
finite over $S$, the sheaves trivialise Zariski locally also over $S$. It
follows that the $\sheaf{A}^\times$-torsors trivialise over $S$, which
concludes the proof.
\end{proof}
\section{The Grothendieck group of stacks}
\label{gro}
In this section, we define the Grothendieck ring of stacks and summarise some
of its basic properties. Most of the results in this section are already
described in \cite{ekedahl2009gg} or are simple consequences of results
described there.

Since it is sometimes convenient to work with the relative Grothendieck ring, we
state the results in this generality. Let $\stack{S}$ be an algebraic stack.
We let $\cStack_{\stack{S}}$ denote the 2-category of finitely presented
algebraic stacks over $\stack{S}$. The subcategory of stacks which are
representable by algebraic spaces over $\stack{S}$ is denoted by
$\cAS_{\stack{S}}$. Although many of the definitions and results would make
sense in a more general setting, we will, for simplicity, assume that the base
$\stack{S}$ is of finite presentation over a field.

\begin{definition}
Let $\stack{S}$ be an algebraic stack. The \term{Grothendieck ring of
algebraic stacks} over $\stack{S}$ is denoted by $\KK_0(\cStack_{\stack{S}})$.
As a group it is presented by generators $\{\stack{X}\}$, being
equivalence classes of objects $\stack{X}$ in $\cStack_k$, subject to the
relations
\begin{enumerate}
\item[\gsi] $\{\stack{X}\} = \{\stack{Z}\} + \{\stack{X} \setminus \stack{Z}\}$
if $\stack{Z}$ is a closed substack of $\stack{X}$,
\item[\gsii] $\{\stack{E}\} = \{\mathbb{A}^n \times \stack{X}\}$ if $\stack{E}
\to \stack{X}$ is a vector bundle of constant rank $n$.
\end{enumerate}
This group is endowed with a natural ring structure, with multiplication defined
by $\{\stack{X}\}\cdot\{\stack{Y}\} := \{\stack{X} \times_\stack{S}
\stack{Y}\}$.
The multiplicative identity is given by the class $1 := \{\ast\}$ of the
base $\stack{S}$, and the additive identity is given by the class $0 :=
\{\emptyset\}$ of the empty space. The class $\{\AA^1_\stack{S}\}$ of the affine
line relative to $\stack{S}$ is called the \term{Lefschetz class} and is denoted
by $\LL$. The axiom GS1 is sometimes called the \term{scissors relation}. The
\term{Grothendieck ring of algebraic spaces}, denoted by
$\KK_0(\cAS_{\stack{S}})$, is defined similarly, but with the category
$\cAS_{\stack{S}}$ taken as the underlying category.
\end{definition}
\noindent
Note that if we work over a a field $k$, we have natural inclusions
$\cVar_k \to \cSch_k \to \cAS_k$ of categories, where $\cVar_k$ denotes
the category of varieties over $k$ and $\cSch_k$ denotes the category of
finite type schemes over $k$. These inclusions induce ring homomorphisms
$$
\KK_0(\cVar_k) \to \KK_0(\cSch_k) \to \KK_0(\cAS_k),
$$
on the corresponding Grothendieck rings, which are all isomorphisms. Also,
in this case the relation GS2 is redundant, as we will see from
Proposition~\ref{gro-as-mul} whose proof does not use this relation. In
particular, the ring $\KK_0(\cAS_k)$ is identical to the usual Grothendieck
ring of varieties.

Next, we describe how the relative Grothendieck rings over different bases are
related to each other. Let $f\colon\stack{S}' \to \stack{S}$ be a
morphism of algebraic stacks.
We get a natural map $f^\ast\colon \K0(\cStack_{\stack{S}}) \to
\K0(\cStack_{\stack{S'}})$ induced by pulling back stacks along $f$.
This is
easily seen to be a ring homomorphism taking $\LL$ in $\K0(\cStack_{\stack{S}})$
to $\LL$ in $\K0(\cStack_{\stack{S'}})$.
Since, by our assumptions, the morphism $f$ is of finite presentation,
we also have a map $f_\ast \colon \K0(\cStack_{\stack{S}'}) \to
\K0(\cStack_{\stack{S}})$ in the other direction.
This is induced by post-composition with $f$.
This is a homomorphism of the underlying groups, and we have a
\term{projection formula}
$$
f_\ast\big(f^\ast(a)\cdot b\big) = a\cdot f_\ast(b),
\qquad a \in \K0(\cStack_{\stack{S}}),
\ b \in \K0(\cStack_{\stack{S}'}),
$$
making $f_\ast$ a $\K0(\cStack_{\stack{S}})$-module homomorphism.
Of course, similar formulae hold if we instead work with
$\K0(\cAS_{\stack{S'}})$ and $\K0(\cAS_{\stack{S}})$, but for $f_\ast$ to be
defined, the morphism $f$ has to be representable by algebraic spaces.
\subsection{Classes of fibrations}
We need several results regarding relations associated to fibrations in the
Grothendieck rings. For convenience, we also include most of the proofs.
\begin{prop}
\label{gro-as-mul}
Let $E \to S$ be a fibration of algebraic spaces with fibre $F$. Assume that
the fibration is associated to a torsor for a special group. Then $\{E\} =
\{F\}$ in $K_0(\cAS_S)$.
\end{prop}
\begin{proof}
We may without loss of generality assume that $E$ and $S$ are reduced.
By our standard assumptions on algebraic spaces,
the schematic locus of $S$ is non-trivial.
Since the fibration is associated to a special group $G$,
there is a generic point in the schematic locus of $S$ over which the
fibration is trivial.
Hence there is an open set $U$ over which
$E_U \to U$ is isomorphic to $U \times F \to U$.
It follows that $\{E_U\} = \{F\}\{U\}$ in $K_0(\cAS_\stack{S})$.
If we let $Z$ be the reduced closed subscheme of $S$ with complement $U$, then
also $E_Z \to Z$ is associated to a $G$-torsor.
Under the hypothesis that $\{E_Z\} = \{F\}\{Z\}$, we may therefore conclude
that indeed $\{E\} = \{F\}\{S\}$ since $\{E\} = \{E_U\}+ \{E_Z\}$ and
$\{S\} = \{U\}+ \{Z\}$ by the scissors relation.
Hence the result follows by noetherian induction on $S$,
the statement for $S = \emptyset$ being trivial.
\end{proof}
\noindent
Note that if $S$ in the proposition above is of finite presentation over
$S_0$ and $F$ defined over $S_0$, then the
projection formula gives the relation $\{E\} = \{F\}\{S\}$ in $\K0(\cAS_{S_0})$.
Note also that Axiom GS2 is not used in the proof. Since GS2 is actually a
consequence of the proposition, we see that the axiom is redundant in this
situation.

\begin{prop}
\label{gro-bgln}
Let $n$ be a natural number and $\stack{T} \to \stack{S}$ be a $\GL_n$-torsor of
algebraic stacks. Then we have the relation $\{\stack{T}\} = \{\GL_n\}\{\stack{S}\}$
with $\{\GL_n\} = \prod_{i = 0}^{n-1}(\LL^n - \LL^i)$ in $\KK_0(\cStack_k)$.
In particular, since $\ast \to \BB\GL_n$ is a $\GL_n$-torsor, we have that
$1 = \{\GL_n\}\{\BB \GL_n\}$, so $\{\GL_n\}$ is invertible in
$\KK_0(\cStack_k)$.
\end{prop}
\begin{proof}
Let $\stack{E} \to \stack{S}$ be the vector bundle associated to the $\GL_n$-torsor
$\stack{T} \to \stack{S}$. For $1 \leq i \leq n$, consider the map $\stack{E}^i
\to \bigwedge^i \stack{E}$ from the $i$-th fibre power to the $i$-th exterior
power over $\stack{S}$ taking an $i$-tuple of sections to their exterior product.
Define the stack $\stack{F}_i$ to be the complement of the
pullback of the zero section along this map in $\stack{E}^i$. Informally, we may
think of this as the stack of $i$-tuples of linearly independent vectors in
$\stack{E}$. In particular $\stack{F}_0 \simeq \stack{S}$ and $\stack{F}_n$,
being the frame bundle of $\stack{E}$, is isomorphic to $\stack{T}$ as
$\GL_n$-torsors over $\stack{S}$.

The stack $\stack{E}\times_\stack{S}\stack{F}_i$ is a rank $n$ vector bundle over
$\stack{F}_i$. The map 
$\stack{E}\times_\stack{S}\stack{F}_i \to \bigwedge^{i+1}
\stack{E}\times_\stack{S}\stack{F}_i$ defined by
$(v, v_1, \ldots, v_i) \mapsto (v\wedge v_1 \wedge \cdots \wedge v_i, (v_1, \ldots, v_i))$
may be viewed as a morphism of $\sheaf{O}_{\stack{F}_i}$-modules. Its kernel is a rank $i$
vector bundle over $\stack{F}_i$ whose complement in
$\stack{E}\times_\stack{S}\stack{F}_i$ is $\stack{F}_{i+1}$. Hence the class of
$\stack{F}_{i+1}$ in $\KK_0(\cStack_k)$ is
$(\LL^n - \LL^i)\cdot\{\stack{F}_i\}$, as seen by using Axiom~GS1 once and
Axiom~GS2 twice. By induction on $i$, we therefore get the relation
$$
\{\stack{T}\} = \prod_{i = 0}^{n-1}(\LL^n - \LL^i)\{\stack{S}\}.
$$
The statement about the class of $\GL_n$ follows from the special case where
$\stack{T} = \GL_n$ and $\stack{S} = \ast$. The statement about
a general $\GL_n$-torsor follows by substituting this back into the displayed
equation.
\end{proof}

Let $\Phi$ denote the set of cyclotomic polynomials in the indeterminate $\LL$
and $\Phi_{\LL}$ denote the set $\Phi \cup \{\LL\}$. Note that inverting the
classes of $\GL_n$ for all natural numbers $n$ in $\K0(\cAS_{\stack{S}})$ is the
same thing as forming the localisation $\Phi_{\LL}^{-1}\K0(\cAS_{\stack{S}})$.
Proposition~\ref{gro-bgln} implies that we have a canonical factorisation
$$
\begin{tikzpicture}[description/.style={fill=white,inner sep=2pt}]
\matrix (m) [matrix of math nodes, row sep=3em,
column sep=2.5em, text height=1.5ex, text depth=0.25ex]
{
\K0(\cAS_{\stack{S}}) & \Phi_{\LL}^{-1}\K0(\cAS_{\stack{S}}) \\
& \K0(\cStack_{\stack{S}}) \\
};
\path[->,font=\scriptsize]
(m-1-1)
edge  (m-1-2)
edge (m-2-2)
(m-1-2)
edge node[auto] {$\eta$} (m-2-2);
\end{tikzpicture}
$$
of the natural map $\K0(\cAS_{\stack{S}}) \to \K0(\cStack_{\stack{S}})$.
As shown in \cite[Theorem~1.2]{ekedahl2009gg}, the map $\eta$ in the diagram
above is an isomorphism.

\begin{prop}
\label{gro-st-loc}
Let $\stack{X} \to \stack{S}$ be a morphism of stacks, and let $C$ be a fixed element
in $\KK_0(\cStack_k)$. Assume that for each morphism $S \to \stack{S}$ with $S$ a
scheme, we have the relation $\{\stack{X}_S\} = C \cdot \{S\}$.
Then we also have the relation $\{\stack{X}\} = C\cdot\{\stack{S}\}$.
\end{prop}
\begin{proof}
If $\stack{Z}$ is a closed substack of $\stack{S}$ with complement
$\stack{U}$, it is enough to give a proof for $\stack{X}_{\stack{Z}} \to
\stack{Z}$ and $\stack{X}_{\stack{U}} \to \stack{U}$ separately. Indeed,
Axiom~GS1 gives us the relations
$\{\stack{X}\} = \{\stack{X}_\stack{Z}\} + \{\stack{X}_\stack{U}\}$
and~$\{\stack{S}\} = \{\stack{Z}\} + \{\stack{U}\}$ which then would
give us the desired result. It follows by noetherian induction that it is
enough to prove that the proposition holds over a non-empty open subset of $\stack{S}$.
Hence we may assume that $\stack{S}$ is a global quotient $[S/\GL_n]$ with $S$ a scheme. 
Consider the 2-cartesian square
$$
\begin{tikzpicture}[description/.style={fill=white,inner sep=2pt}]
\matrix (m) [matrix of math nodes, row sep=3em,
column sep=2.5em, text height=1.5ex, text depth=0.25ex]
{
\stack{X}_S & \stack{X} \\
S & \stack{S}. \\
};
\path[->,font=\scriptsize]
(m-1-1)
edge  (m-2-1)
edge (m-1-2)
(m-1-2)
edge (m-2-2)
(m-2-1)
edge (m-2-2);
\end{tikzpicture}
$$
The horizontal arrows are $\GL_n$-torsors, so by Lemma~\ref{gro-bgln}, we get
the identities $\{S\} = \{\GL_n\}\{\stack{S}\}$
and~$\{\stack{X}_S\} = \{\GL_n\}\{\stack{X}\}$.
Combining these relations with the hypothesis
about pullbacks to schemes gives
$\{\GL_n\}\{\stack{X}\} = C \cdot \{\GL_n\}\{\stack{S}\}$. Since the factor
$\{\GL_n\}$ is invertible, we may cancel it to get the desired result.
\end{proof}

\begin{corollary}
\label{gro-st-mul}
Let $G$ be a special group and let $\stack{T} \to \stack{S}$ be a $G$-torsor of
algebraic stacks. Then we have the relation 
$
\{\stack{T}\} = \{G\}\{\stack{S}\}
$
in $K_0(\cStack_k)$.
In particular, since $\ast \to \BB G$ is a $G$-torsor, we have that
$1 = \{G\}\{\BB G\}$, so $\{G\}$ is invertible in $\KK_0(\cStack_k)$.
Furthermore, if $F$ is an algebraic $G$-space and
$\stack{E} \to \stack{S}$ is an $F$-fibration associated to a torsor
as above, then $\{\stack{E}\} = \{F\}\{\stack{S}\}$. 
\end{corollary}
\begin{proof}
This is a direct application of Proposition~\ref{gro-st-loc} to
Proposition~\ref{gro-as-mul}.
\end{proof}

\begin{corollary}
\label{gro-st-fact}
Let $G \to H$ be a homomorphism of algebraic groups with $H$ special, and let
$F$ be an algebraic $G$-space with its $G$-action factoring through $H$. Assume
that $\stack{X} \to \stack{S}$ is an $F$-fibration of stacks associated to a
$G$-torsor. Then $\{\stack{X}\} = \{F\}\{\stack{S}\}$ in $\KK_0(\cStack_k)$.
\end{corollary}
\begin{proof}
Denote the $G$-torsor by $\stack{T} \to \stack{S}$. Since the action of $G$ on
$F$ factors through $H$, we may view $F$ as an $H$-space, which we denote
${}_HF$. Then we get a natural identification $F \cong {}_GH_H \contr{H}
{}_HF$, where ${}_GH_H$ is just the group $H$ regarded as a $(G, H)$-space.
The fibration $\stack{X}$ is obtained by taking the contraction product
$\stack{T} \contr{G} ({}_GH_H \contr{H} {}_HF)$. Associativity of the
contraction product gives that $\stack{X}$ is equivalent to the
${}_HF$-fibration associated to the $H$-torsor $\stack{T} \contr{G} {}_GH_H$.
Since $H$ is special, the result follows from
Corollary~\ref{gro-st-mul}.
\end{proof}

\subsection{Expected classes}
\label{gro-expect}
By taking Euler characteristics with compact support, one can get a heuristic
assessment of what a class in $\K0(\cAS_\stack{S})$ or $\K0(\cStack_\stack{S})$
ought to be. In \cite[Section~2]{ekedahl2009gg}, Ekedahl uses a generalised
Euler characteristics taking values in the Grothendieck ring $\K0(\Coh_k)$ of mixed
Galois representations. This yields a ring homomorphism
$\ec\colon\K0(\cAS_{k}) \to \K0(\Coh_k)$. We also get an induced ring
homomorphism $\ec\colon\K0(\cStack_{k}) \to \Phi_\LL^{-1}\K0(\Coh_k)$, which we
denote by the same symbol. We also use the symbol $\LL$ for the image of $\LL$
in $\K0(\Coh_k)$. Since the ring $\K0(\Coh_k)$ contains the polynomial ring
$\ZZ[\LL]$, it follows the same holds for $\K0(\cAS_k)$.
Given a stack $\stack{S}$, we can take a point
$\xi\colon\Spec k \to \stack{S}$.
The existence of the homomorphism
$\xi^\ast\colon \K0(\cAS_\stack{S}) \to \K0(\cAS_k)$ shows that also 
$\K0(\cAS_\stack{S})$ contains $\ZZ[\LL]$, and, by localising, that
$\K0(\cStack_{\stack{S}})$ contains $\Phi_{\LL}^{-1}\ZZ[\LL]$.

Let $X$ be an object in $\cStack_k$. Given that $\ec(a)$ lies in the subring
$\Phi_\LL^{-1}\ZZ[\LL]$ of $\Phi_\LL^{-1}\K0(\Coh_k)$, we call the corresponding
class in $\K0(\cStack_k)$ the \term{expected class} of $X$. As mentioned in the
introduction, the actual class sometimes differs from the expected class.
Indeed, for a connected group $G$, we always have $\ec(T) = \ec(G)\ec(S)$ for
$G$-torsors $T\to S$, whereas the corresponding relation in $\K0(\cStack_k)$
does not necessarily hold for non-special $G$. For a finite group $G$, we always
have $\ec(\BB G) = 1$ (see \cite[Proposition~3.1]{ekedahl2009fg}). For symmetric
groups, the expected value is always assumed
\cite[Theorem~4.3]{ekedahl2009fg}.
\begin{prop}[\thmauth{Ekedahl}]
\label{gro-bsn}
For each field $k$ and each symmetric group $\Sigma_n$, we have $\{\BB
\Sigma_n\} = 1$ in $\K0(\cStack_k)$.
\end{prop}
\noindent
Ekedahl also gives examples of finite groups for which the class of the
classifying stack does not assume the expected value. For instance, we have
$\{\BB\ZZ/47\ZZ\} \neq 1$ in $\K0(\cStack_\QQ)$
(see~\cite[Corollary~5.8]{ekedahl2009fg}). However, for many smaller cyclic
groups, the class is the expected one (see~\cite[Section~3]{ekedahl2009fg}).
\begin{prop}[\thmauth{Ekedahl}]
\label{gro-bcyclic}
Let $k$ be an arbitrary field, and $p$ a prime number from the set $\{2,
3, 5, 7, 11\}$. Then $\{\BB \ZZ/p\ZZ\} = 1$ in $\K0(\cStack_k)$.
\end{prop}

\subsection{Computing the class of a classifying stack}
\label{gro-com}
section which frequently will be used in actual computations. We also make
some sample computations.
\begin{prop}
\label{gro-rep}
Let $G$ be an affine group over a field $k$ acting linearly on an
$n$-dimensional $k$-vector space $V$. Then we have the relations
$$
\{\BB G\} = \{[V/G]\}\cdot \LL^{-n} = \{[\PP(V)/G]\}\cdot\frac{\LL - 1}{\LL^n - 1}
$$
in $\KK_0(\cStack_k)$.
\end{prop}
\begin{proof}
For the first equality, we apply Corollary~\ref{gro-st-fact},
with $\GL_n$ as our special group, the space $V$ as our fibre and
the 1-morphism $[V/G] \to \BB G$ as our $V$-fibration associated to
the $G$-torsor $\ast \to \BB G$. For the second equality, we instead have
the fibre $\PP(V)$ and the $\PP(V)$-fibration $[\PP(V)/G] \to \BB G$
associated to the same torsor as above. 
\end{proof}
It should be stressed that it is crucial that the action of $G$
on $\PP(V)$ comes from a linear action on $V$ in the proposition above.

\begin{prop}
\label{gro-exact}
Let $1 \to G \to E \to K \to 1$ be an exact sequence of algebraic space
groups, flat over an algebraic stack $\stack{S}$, with $E$ special. Then we have
the relation
$$
\{\BB_{\stack{S}} G\} = \{K\}/\{E\}
$$
in $\KK_0(\cStack_\stack{S})$.
\end{prop}
\begin{proof}
The action of $E$ on $K$ by left translation gives an $E$-torsor $K \to [K/E]$,
so $\{K\} = \{E\}\{[K/E]\}$ by Corollary~\ref{gro-st-mul}.
By the same corollary, we know that the class $\{E\}$ is invertible.
Hence the result follows follows from the fact that the stack $[K/E]$ is
equivalent to $\BB_{\stack{S}} G$.
\end{proof}

\begin{remark}
Note that since $E$ is special, it has affine fibres. This property
is stable under taking closed subgroups, so the same is true for $G$.
Since $G$ is assumed to be flat, and also of finite presentation by
our default assumption, over $\stack{S}$, it follows that its classifying stack
is algebraic with affine stabilisers. Hence the statement above makes sense.
\end{remark}

As a direct application of Proposition~\ref{gro-exact}, we compute the class
of the classifying stack of the group of $n$-th roots of unity.
\begin{prop}
\label{gro-units}
The class of the classifying stack $\BB \mu_n$ for the group of $n$-th roots of
unity is trivial in $\KK_0(\cStack_k)$ for any field $k$.
\end{prop}
\begin{proof}
Consider the Kummer sequence
$
1 \to \mu_n \to \GGm \to \GGm \to 1.
$
Since $\GG_m$ is special, the statement follows from
Proposition~\ref{gro-exact}.
\end{proof}

For more complicated groups $G$, it is harder to compute the class of
the classifying stack $\BB G$.
Our strategy will be to find a linear representation $V$ and invoke
Corollary~\ref{gro-rep}. This reduces the problem to computing the
class of the stack $[V/G]$. Stratifying $V$ into locally closed
$G$-invariant subschemes allows us to reduce the problem further.

We end the section by working out some examples using the techniques
described in this section. The results will be used in the next section.

\begin{prop}
\label{gro-class-gm2}
Consider the group $G = \GGm \rtimes \Sigma_2$, with $\Sigma_2$ acting as the
automorphism group of $\GGm$. If 2 is invertible in the field $k$, then
$\{\BB G\} = \LL(\LL^2-1)^{-1}$ in $\KK_0(\cStack_k)$.
\end{prop}
\begin{proof}
To see this, consider the following action of $G$ on $\PP^1$. The subgroup
$\GGm$ acts by multiplication on the first homogeneous coordinate and by
multiplication with the inverse on the second. The subgroup $\Sigma_2$ acts by
permuting the homogeneous coordinates. This action obviously comes from a linear
action, so we may apply Proposition~\ref{gro-rep} and get
$\{\BB G\} = \{[\PP^1/G]\}(\LL + 1)^{-1}$.

The $G$-space $\PP^1$ has two orbits. A closed orbit containing the point 
$(1 \hcolon 0)$ and an open orbit containing $(1 \hcolon 1)$. The stabilisers
of these points are $\GG_m$ and $\mu_2 \times \Sigma_2$ respectively.
Stratifying the stack $[\PP^1/G]$ according to these orbits gives the relation 
$\{[\PP^1/G]\}  =  \{\BB \GGm\} + \{\BB (\mu_2 \times \Sigma_2)\}$.
Since $\GGm$ is special, we have $\{\BB \GGm\} = (\LL-1)^{-1}$. Furthermore
$\mu_2 \times \Sigma_2 = \mu_2 \times \mu_2$ under our assumptions on the base
field, so $\BB (\mu_2 \times \Sigma_2)$ is isomorphic to
$\BB \mu_2 \times \BB \mu_2$ which has class~1 according to
Proposition~\ref{gro-units}. Combining these relations gives the desired
result.
\end{proof}

\begin{prop}
\label{gro-class-u2}
Consider the subgroup $G = \mu_n \rtimes \Sigma_2$ of the group
$\GGm \rtimes \Sigma_2$ from the previous proposition. Assume that 2 is
invertible in the field $k$ and that $4$ does not divide $n$.
Then $\{\BB G\} = 1$ in $\KK_0(\cStack_k)$.
\end{prop}
\begin{proof}
If $n = 2q$, with $q$ odd, then we have an isomorphism
$G \simeq \mu_2 \times (\mu_q \rtimes \Sigma_2)$ which
gives $\BB G \simeq \BB \mu_2 \times \BB (\mu_q \rtimes \Sigma_2)$. This reduces
the problem to the case when $n$ is odd.

By using the same representation of $G$ on $\PP^1$ as in the proof of
the previous proposition, we get the relation $\{\BB G\} =
\{[\PP^1/G]\}(\LL + 1)^{-1}$. As before, we may also isolate the closed orbit
$\{0, \infty\}$ to get the relation $\{[\PP^1/G]\}  = 
\{\BB \mu_n\} + \{[U/G]\} = 1 + \{[U/G]\}$, where $U$ denotes the complement.
Since the subgroup $\mu_n$ acts freely on $U$, we have an isomorphism $[U/G] \cong
[(U/\mu_n)/\Sigma_2]$. Here $U/\mu_n \cong \GG_m$ on which $\Sigma_2$
acts as the automorphism group of $\GG_m$. This action has two fixed points,
namely $\pm 1$. The quotient $\GGm/\Sigma_2$ is isomorphic to $\AA^1$.
Therefore, the quotient $(\GGm \setminus \{\pm1\})/\Sigma_2$ is isomorphic to
$\AA^1$ minus two points. Hence the usual stratification argument gives
$\{[\GGm/\Sigma_2]\} = 2\{\BB \Sigma_2\} + \LL-2 = \LL$. The result follows by
combining the relations.
\end{proof}

\section{Universal tori}
\label{tori}
In this section, we investigate certain quasi-split tori over stacks. We will
also look at a larger class of tori, which we call stably rational in analogy
with the situation where we work over a field.
\begin{definition}
Let $G$ be a finite group and $E$ a finite $G$-set. The torus
$$
\Res_{[E/G]/\BB G}\GGm
$$
obtained as the Weil restriction of $\GGm$ along the natural map $[E/G] \to \BB
G$ is called the \term{universal} quasi-split torus associated to the $G$-set $E$.
In the special case when $E$ is the $\Sigma_n$-set $[n]$, which makes $[E/G]$
isomorphic to $\BB \Sigma_{n-1}$, we simply call the associated torus
$\Res_{\BB \Sigma_{n-1}/\BB \Sigma_n}\GGm$ the universal quasi-split torus of rank $n$.
\end{definition}
\noindent
The universal quasi-split torus of rank $n$ can also be described as the
stack-quotient $[\GGm^n/\Sigma_n]$, where $\Sigma_n$ acts on $\GGm^n$ by permuting the factors.
It is universal in the sense that any rank $n$ quasi-split torus can be obtained
from it via base change.
Indeed, for any finite, étale morphism $S' \to S$ of degree $n$,
we may consider the configuration space $\Conf^n S'/S$.
This is, by definition, the $n$-fold fibre-product of $S'$ over $S$
with the big diagonal removed.
The space $\Conf^n S'/S$ is a $\Sigma_n$-torsor over $S$,
and therefore corresponds to a morphism $S \to \BB \Sigma_n$
to the classifying stack for $\Sigma_n$.
The torus $\Res_{S'/S}\GGm$ is isomorphic to the pull-back of
of the rank $n$ universal quasi-split
torus $\Res_{\BB\Sigma_{n-1}/\BB\Sigma_{n}}\GGm$ along this morphism.

A torus $T$ over a field is called \term{stably rational} if there exists a split
torus $\GGm^n$ such that $T \times \GGm^n$ is a rational variety. Due to a
theorem by Voskresenski{\u\i} \cite{voskresenskii1973}, a torus over a field is
stably rational if and only if it fits into an exact sequence of tori
$$
1 \to T_1 \to T_2 \to T \to 1
$$
with $T_1$ and $T_2$ quasi-split. This motivates the following definition.
\begin{definition}
A torus $T$ over an arbitrary base $S$ is called \term{stably rational} if it
fits into an exact sequence of tori
$$
1 \to T_1 \to T_2 \to T \to 1
$$
with $T_1$ and $T_2$ quasi-split over $S$.
\end{definition}
Since quasi-split tori are special, we can easily express the classes of their classifying
stacks in terms of the classes of the tori themselves.
\begin{prop}
\label{mon-class}
Let $T$ be a stably rational torus over a base $S$, which might be an algebraic stack,
fitting into the exact sequence
$$
1 \to T_1 \to T_2 \to T \to 1,
$$
with $T_1$ and $T_2$ quasi-split. Then $\{T\} = \{T_2\}/\{T_1\}$ and
$\{\BB T^\circ\} = \{T_1\}/\{T_2\}$ in the ring $\KK_0(\cStack_S)$.
Here $T^\circ$ denotes the torus with character lattice dual to that of $T$.
\end{prop}
\begin{proof}
The first statement follows by Corollary~\ref{gro-st-mul} since $T_2 \to T$ is a $T_1$-torsor and $T_1$
is special. The second statement follows by dualising the sequence and applying Proposition~\ref{gro-exact},
using the fact that quasi-split tori are self dual.
\end{proof}

It is possible to compute the class of a quasi-split torus in the Grothendieck ring of
algebraic spaces explicitly. This has been done by in \cite{rokeaus2007} for
tori over fields and in \cite{bergh2014} in the relative setting. The class
of a quasi-split torus may be expressed in terms
of classes in the \term{Burnside ring}.
A standard reference regarding Burnside rings,
the ring of super central functions and their lambda ring structures is
\cite{knutson1973}.
We briefly recall the theory we need below.

Let $G$ be a finite group. The \term{Burnside ring} $\Bu{G}$ is the Grothendieck
ring associated to the category of finite $G$-sets. The ring $\Bu{G}$ has a
natural structure of pre-lambda ring induced by symmetric powers as follows. The
class of the $i$-th symmetric power of a $G$-set $E$ in $\Bu{G}$ is denoted by
$\sigma^i(E)$, and we define the power series $\sigma_s(E) = \sum_{i \geq 0}
\sigma^i(E)s^i$. The operation $\sigma_s$ extends to a pre-lambda ring structure
on $\Bu{G}$. The dual operator is denoted by $\lambda_s$ and is defined by the
relation $\lambda_{-s}\sigma_s = 1$.

Let $k$ be a field and $G$ be a finite group. Then there is a natural
ring homomorphism $\Bu{G} \to \KK(\cAS_{\BB G})$ from the Burnside ring
to the relative Grothendieck group over the classifying stack $\BB_k G$.
The homomorphism takes the class of a $G$-set $E$ to the class of the stack
quotient $[E/G] \to \BB G$. We will also consider the extension
$\Bu{G}[\ell] \to \KK(\cAS_{\BB G})$ of the map
to the polynomial ring $\Bu{G}[\ell]$ where $\ell$ maps to the
class of $\LL_{\BB G} \to \BB G$. We state the theorem about the explicit
expression for the class of a quasi-split torus in this setting.

\begin{prop}
\label{prop-class-qs-torus}
Let $G$ be a finite group and $E$ a finite $G$-set with $n$ elements. Then the
class of the universal quasi-split torus $\Res_{[E/G]/\BB G}\GGm$ in
$\KK(\cAS_{\BB G})$ is the image of the element
$$
(\ell - 1)^E := \sum_{i = 0}^{n} (-1)^i\lambda^i(E)\ell^{n-i}
$$
in $\Bu{G}[\ell]$ under the ring homomorphism $A(G)[\ell] \to
\KK(\cAS_{\BB G})$ described in the paragraph preceding the proposition.
\end{prop}
\noindent
Here we treat $(\ell - 1)^E$ just as short hand notation for the right hand
side.
If $G$ acts trivially, we have the equality $(\ell - 1)^E = (\ell - 1)^{|E|}$,
which motivates the notation.

The original theorem regarding the class of the torus $L^\times$ for
a degree $n$ separable algebra extension $L/k$ is recovered by letting
$f\colon \Spec k \to \BB \Sigma_n$ be the map corresponding to the
$\BB\Sigma_n$-torsor $\Conf^n L/k$ and applying the ring homomorphism
$f^\ast\colon\K0(\cAS_{\BB\Sigma_n}) \to \K0(\cVar_k) $, which respects the
pre-lambda structure, to the class of the rank $n$ universal torus.

Computations in the Burnside ring can often be made very explicit with the
use of the ring $\SCF{G}$ of \term{super central functions} for $G$.
The elements of $\SCF{G}$ are integer valued class functions on the
set of subgroups of $G$. That is, functions which are constant on
conjugacy classes.
Addition and multiplication is defined elementwise.
Given a $G$-set $E$, and a subgroup $H \subset G$, the \term{$H$-mark} of $E$ is
defined as the number of fixed points of $E$ under the natural $H$-action on $E$
obtained by restriction.
Since the number of fixed points only depends on the conjugacy class of $H$,
taking marks for all subgroups of $G$ gives a super central function.
This induces a map $\Bu{G} \to \SCF{G}$,
which turns out to be a ring homomorphism.
By a classical result by Burnside~\cite[§180]{burnside1955},
this homomorphism is injective.
The pre-lambda ring structure on $\Bu{G}$ extends to $\SCF{G}$ in natural way.
We also get an injective function $\Bu{G}[\ell] \to \SCF{G}[\ell]$ of polynomials.
Using this, we get a more explicit description of the expression $(\ell - 1)^E$
in terms of marks.

\begin{prop}
Let $G$ be a finite group and $E$ a finite $G$-set. Let $H \subset G$ be a
subgroup and let $w_1, \ldots, w_r$ be the length of the orbits of $E$ under the
natural $H$-action.
Then the $H$-mark of $(\ell-1)^E$ is given by the polynomial
$$
(\ell^{w_1} - 1)\cdots(\ell^{w_r} - 1)
$$
in $\ZZ[\ell]$.
\end{prop}
\begin{proof}
\label{prop-qs-as-scf}
Since the pre-lambda ring structure is functorial with respect to
restriction, it is enough to consider the case when $H = G$.
First we prove that the $G$-marks of $\sigma_s(E)$ is
given by the power series
$$
\frac{1}{(1 - s^{w_1})\cdots(1 - s^{w_r})}.
$$
By the relation $\sigma_s(E + F) = \sigma_s(E)\sigma_s(F)$ for $G$-sets $E$ and
$F$, it is enough to consider the case when $G$ acts transitively on $E$. Denote
the cardinality of $E$ by $w$.
An unordered tuple $(x_1, \ldots, x_n) \in \Symm_n(E)$ is a fixed point if and
only if each element of $E$ occurs in the tuple an equal number of times. Hence
we have one fixed point if $w | n$ and zero fixed points otherwise. This gives
us the $G$-mark $1/(1-s^w)$ as desired.

Since $\lambda_{-s}(E)\sigma_s(E) = 1$, it follows that
$$
\sum_{i = 0}^{|E|} (-1)^i\lambda^i(E)s^i = (1 - s^{w_1})\cdots(1 - s^{w_r}).
$$
Substituting $\ell$ for $s^{-1}$ and multiplying with $\ell^{|E|}$ gives the
desired identity.
\end{proof}

Since we also want to use the Burnside ring to compute the classes of stably rational tori
in $\KK_0(\cStack_{\BB G})$, we need to invert the classes of the form $(\ell -
1)^E$. It turns out that it is enough to invert the set $\Phi \subset
\Bu{G}[\ell]$ of cyclotomic polynomials in the variable $\ell$.

\begin{prop}
\label{mon-inverse}
Let $G$ be a finite group and $E$ a finite $G$-set. Then the class
$$
(\ell - 1)^E := \sum_{i = 0}^n (-1)^i\lambda^i(E)\ell^{n-i}
$$
in $\Bu{G}[\ell]$ is invertible in $\Phi^{-1}\Bu{G}[\ell]$.
\end{prop}
\begin{proof}
It is enough to prove that given an element
$(\ell - 1)^S = f = a_0 + a_1\ell + \cdots + a_s\ell^s$,
there exists an element $g = b_0 + b_1\ell + \cdots + b_r\ell^r$
in $\Bu{G}[\ell]$ and a cyclotomic polynomial
$h = c_0 + c_1\ell + \cdots + c_{r+s}\ell^{r+s}$ such that $fg = h$.
We first consider the elements $f$, $g$ and $h$ as elements in $\SCF{G}[\ell]$.
Due to the  explicit description in Proposition~\ref{prop-qs-as-scf}, the
polynomial $f_H \in \ZZ[\ell]$ is cyclotomic for all subgroups $H \subset G$.
Let $h$ be the least common multiple of all $f_H$ where
we let $H$ range over a set of representatives for the conjugacy classes of the subgroups of $G$.
Now choose $g$ such that $g_H = h/f_h$ for all subsets $H$ of $G$. We need to prove that $g$ actually
comes from an element of $\Bu{G}[\ell]$. But this follows from the relations
$c_i = \sum_{i = j + k}a_jb_k$. Indeed, since $a_0$ is invertible in $\Bu{G}$, we can solve for
$b_i$, and thus express $b_i$ in terms of elements which lie in $\Bu{G}$ by induction on $i$.
\end{proof}

\section{Groups of monomial matrices}
A monomial matrix is a square matrix with exactly one non-zero element
in each row and each column. Fix a positive integer $n$.
We denote the subgroup of monomial matrices in $\GL_n$ by $\Mon_n$.
By taking the quotient by the scalar matrices, we get a corresponding subgroup
$\PMon_n$ of $\PGL_n$. The groups $\PMon_n$ and $\Mon_n$ are the normalisers of
the groups of diagonal matrices in $\PGL_n$ and $\GL_n$ respectively.

\begin{prop}
\label{monomial-class}
Let $\Mon_n$ be the subgroup of monomial matrices in $\GL_n$ and let $X_n$ be
the space $\GL_n/\Mon_n$. Then $\{X_n\} = \LL^{n(n-1)}$ and $\{\BB \Mon_n\} =
\{X_n\}/\{\GL_n\}$ in $\K0(\cStack_k)$ for an arbitrary field $k$.
\end{prop}
\begin{proof}
Since the space $X_n$ is a $\GL_n$-torsor over $\BB \Mon_n$ the two statements
in the proposition are equivalent. Consider the usual linear action of $\Mon_n$
on $V = k^n$. The action has $n+1$ orbits $V_i$, $0 \leq i \leq n$, with $V_i$
being the subvariety of $V$ consisting of points with exactly $i$ of the coordinates vanishing.
This gives a stratification of the stack $[V/\Mon_n]$ in locally
closed substacks $[V_k/\Mon_n]$ for $0 \leq i \leq n$. Since the stabiliser
of a point in $V_i$ is isomorphic to $\Mon_i \times \Sigma_{n-i}$, we have
isomorphisms $[V_i/\Mon_n] \simeq \BB \Mon_i \times \BB \Sigma_{n-i}$.
By applying Proposition~\ref{gro-rep}, we get the recurrence relation
$$
\{\BB \Mon_n\}\LL^n = \{\BB N_0\}\{\BB \Sigma_n\}
+ \{\BB \Mon_1\}\{\BB \Sigma_{n-1}\} 
+ \cdots + \{\BB \Mon_n\}\{\BB \Sigma_0\}.
$$
By Proposition~\ref{gro-bsn}, we have $\{\BB \Sigma_i\} = 1$ for all $i$. 
Solving the equation above, using the recurrence relation
$\{\GL_{i+1}\} = (\LL^{i+1}-1) \LL^i\{\GL_{i}\}$
for the class of the general linear group, yields the desired expression
for $\{\BB\Mon_n\}$ in $\K0(\cStack_k)$.
\end{proof}

The quotient space $X_n = \GL_n/\Mon_n$ is the configuration space
of $n$ unordered points in general position in $\PP^{n-1}$. This space can
also be described as the quotient $\PGL_n/\PMon_n$, making it a $\PGL_n$-torsor over
$\BB \PMon_n$. Hence the expected class of $\BB \PMon_n$ in $\K0(\cStack_k)$ is
$\LL^{n(n-1)}/\{\PGL_n\}$. However, to determine whether this really is the actual
class, seems to be much harder than computing the class of $\{\BB \Mon_n\}$. For
the case $n = 2$, we have already done it in the previous section, since $\PMon_2$ is
isomorphic to $\GG_m \rtimes \Sigma_2$. In general, the problem is related to
determining the class of the classifying stack of a stably rational torus.

We first describe the problem from a slightly more general perspective.
Let $N$ be an algebraic group and $G$ a finite group acting on $N$ by group isomorphisms.
Then we have a split exact sequence $1 \to N \to N \rtimes G \to G \to 1$, which induces
a 2-cartesian square
\begin{center}
\begin{tikzpicture}[description/.style={fill=white,inner sep=2pt}]
\matrix (m) [matrix of math nodes, row sep=3em,
column sep=2.5em, text height=1.5ex, text depth=0.25ex]
{
\BB N & \BB N \rtimes G \\
\ast & \BB G. \\
};
\path[->,font=\scriptsize]
(m-1-1)
edge (m-2-1)
(m-1-2)
edge (m-2-2)
(m-1-1)
edge (m-1-2)
(m-2-1)
edge (m-2-2)
;
\end{tikzpicture}
\end{center}
of classifying stacks. From this, we see that the map $f\colon\BB N \rtimes G
\to \BB G$ is a gerbe, which is neutral since the exact sequence of groups is
split. In fact, the gerbe can be viewed as the classifying stack $\BB_{\BB G}[N/G]$ associated to the group
object $[N/G]$ over $\BB G$.

Applying this to the group $\Mon_n$ of monomial matrices in $\GL_n$, we see that
$\BB \Mon_n$ is isomorphic to the classifying stack of
$\Res_{\BB\Sigma_{n-1}/\BB\Sigma_n}\GGm$ over $\BB\Sigma_n$.
Similarly, the classifying stack for the group $\PMon_n$ is
isomorphic to the classifying stack of the quotient
$(\Res_{\BB\Sigma_{n-1}/\BB\Sigma_n}\GGm)/\GG_m$ over $\BB \Sigma_n$. 
This allows us to use the methods from the previous section to compute class
of $\BB \PMon_n$ for all $n$ for which the dual of the group
$(\Res_{\BB\Sigma_{n-1}/\BB\Sigma_n}\GGm)/\GG_m$ is stably rational.

\begin{proof}[%
\ifarxiv%
Proof %
\fi%
of Theorem~\ref{mainBPN}]
Assume that $n = 2$ or $n= 3$. According to the discussion above, the stack $\BB
\PMon_n$ is the classifying stack of a $(n-1)$-dimensional torus over
$\BB\Sigma_n$. The dual of this torus is stably rational, since all tori of
dimension less or equal than 2 are stably rational
\cite[§4.9]{voskresenskii1997}.
Now we apply Proposition~\ref{mon-class}, \ref{prop-class-qs-torus}
and~\ref{mon-inverse}. It follows that its classifying stack has a class of
the form $f/g$ in $\KK_0(\cStack_{\BB_k\Sigma_n})$, where $f$
is a polynomial in $\LL$ with coefficients in the image of $\Bu{\Sigma_n}$, and
$g$ is a product of cyclotomic polynomials. Since $1/g$ is defined
in $\KK_0(\cStack_k)$, the push forward of
the class $f/g$ along $\BB_k\Sigma_n \to \Spec k$ can also be described as $f/g$
in $\KK_0(\cStack_k)$ by the projection formula. By the assumption on $n$,
we have $\{\BB H\} = 1$ in $\KK_0(\cStack_k)$ for each subgroup $H \subset
\Sigma_n$ according to Proposition~\ref{gro-bsn} and~\ref{gro-bcyclic}.
Hence $f$ has integer coefficients in $\KK_0(\cStack_k)$, so the quotient $f/g$
is rational in $\LL$. As discussed in Section~\ref{gro-expect}, it follows that
the class must coincide with the expected class, which we have seen is
$\LL^{n(n-1)}/\{\PGL_n\}$, and Theorem~\ref{mainBPN} follows.
\end{proof}

The method in the proof fails already for $n = 4$. Indeed, let $V$ be the
Klein four group, and consider the dual of $(\Res_{\ast/\BB V}\GGm)/\GG_m$.
This torus is not stably rational by \cite[§4.10]{voskresenskii1997}.
Hence also the dual of $(\Res_{\BB\Sigma_{3}/\BB\Sigma_4}\GGm)/\GG_m$ fails to
be stably rational, since the former torus can be obtained as a base change of
the latter.

\section{Projective linear groups}
\label{com}
In this section, we start our investigation of the classes of
$\BB\PGL_n$ in the Grothendieck group $\KK_0(\cStack_k)$. The direct approach
described here is only feasible for $n = 2, 3$. We start by outlining the general approach.

Let $V$ be an $n$-dimensional vector space and denote by
$H^{d, n} := \PP\left((\SS^dV)^\vee\right)$ the space of degree $d$
hypersurfaces in $\PP(V)$. Since $\PGL_n$ is the
automorphism group of $\PP(V)$, we get an induced action by $\PGL_n$ on
$H^{d, n}$. In the cases when $d$ is a power of $n$, this action comes from
a {\em linear} representation of $\PGL_n$. Indeed, assume that $d = n^s$ and
consider the action of $\GL(V)$ on the space $(\SS^dV)^\vee$ of $d$-forms
given by
$$
\alpha \cdot f = v \mapsto (\det \alpha)^sf(\alpha^{-1}(v)),
\qquad \alpha \in \GL(V),
\qquad f \in  (\SS^dV)^\vee.
$$
Since the centre of $\GL(V)$ acts trivially, we get a linear representation of
$\PGL_n$ on the space $(\SS^dV)^\vee$, and $H^{d, n}$ is its projectivisation.

The fact that the $\PGL_n$-action on $H^n := H^{n, n}$ is induced by a
linear representation allows us to apply Proposition~\ref{gro-rep}, which
gives
$$
\{\BB \PGL_n\} = \{[H^n/\PGL_n]\}\frac{\LL-1}{\LL^r - 1},\qquad r = \binom{2n - 1}{n}
$$
This reduces the problem of computing the class of $\BB \PGL_n$ to computing
the class of the stack quotient $[H^n/\PGL_n]$.

Denote by $\Hsing^n$ the closed subspace of $H^n$ of singular hypersurfaces and
let $\Hns^n$ denote its complement. Since smoothness is invariant under projective
equivalence, this gives a stratification of $[H^n/\PGL_n]$ into corresponding
substacks $[\Hns^n/\PGL_n]$ and $[\Hsing^n/\PGL_n]$.

\subsection{\texorpdfstring{The class of $\BB\PGL_2$}{The class of BPGL2}}
\label{com-pg2}
Now we prove Theorem~\ref{mainBPGL} in the case $n=2$. In this case, the space
$H^2$ parametrises hypersurfaces of degree 2 in $\PP^1$. In order to avoid
non-reduced stabilisers, we assume that $2$ is invertible in the base field $k$.

The spaces $\Hns^2$ and $\Hsing^2$ consist of one orbit each. Let $xy$ and~$x^2$ be
representatives for these orbits and let $G_{xy}$ and $G_{x^2}$ denote the corresponding
stabilisers.

We prove that the group $G_{x^2}$ is isomorphic to $\GGa \rtimes \GGm$. This
follows if we prove that the stabiliser of the corresponding action of $\GL_2$
is the subgroup of lower triangular matrices. This is easily seen to be true on
geometric points. Furthermore, a general element $I + \varepsilon(a_{ij})$ of
$\Lie \GL_2$ takes the form $x^2$ to $x^2 + 2x\varepsilon(a_{11}x + a_{12}y)$. Since we are in
characteristic $\neq 2$, this forces $a_{12} = 0$ for an element of the
stabiliser of $x^2$. Hence the dimensions of the Lie algebra and the group coincides.
This proves that the stabiliser is smooth and so is determined by its points.

The stabiliser of $xy$ in $\GL_2$ is the subgroup $\Mon_2$ of monomial matrices.
As in the previous case, this is first verified on points. A similar Lie-algebra
computation as above, gives that the stabiliser is smooth regardless of the
characteristic of the field. Taking the quotient with the scalar matrices
gives $G_{xy} = \PMon_2$.

The group $G_{x^2} = \GGa \rtimes \GGm$ is special, so the class of its
classifying stack is simply the inverse $(\LL(\LL - 1))^{-1}$ of the
class of the group by Corollary~\ref{gro-st-mul}. The class of $\BB \PMon_2$ is
$\LL(\LL^2-1)^{-1}$ by Theorem~\ref{mainBPN}.
Combining these results gives the expression
$$
\left(
  \frac{1}{\LL(\LL - 1)} +
  \frac{\LL}{\LL^2-1}
\right)
\frac{\LL-1}{\LL^3-1}
= \frac{1}{\LL(\LL^2-1)}
$$
for the class of $\BB \PGL_2$. This is indeed the inverse of the
class of $\PGL_2$.

\subsection{The classes corresponding to singular plane cubics}
\label{com-sin}
In this section we compute the class of $[\Hsing^3/\PGL_3]$. In order to avoid
reduced stabilisers, we assume that $6$ is invertible in the base field $k$.
The orbits of $\Hsing^3$ correspond to the eight singular cubics in $\PP^2$
listed in Appendix~\ref{sin}. This gives a stratification of the stack quotient
$[\Hsing^3/\PGL_3]$ into eight locally closed subspaces. Since each orbit of
$\Hsing^3$ contains a rational point, the strata of $[\Hsing^3/\PGL_3]$ are
isomorphic to classifying stacks of the stabilisers of the respective points.
Again, according to the appendix, these stabiliser groups are

\vspace{8 pt}
\hfill
\hbox{%
\begin{tabular}{llllllll}
{a)} & $\GGa^2 \rtimes \GL_2$, &
{b)} & $\GGa^2 \rtimes \GGm^2$, &
{c)} & $\GGa^2 \rtimes G$, &
{d)} & $\PMon_3$, \\
{e)} & $\PMon_2$, &
{f)} & $\GGa \rtimes \GGm$, &
{g)} & $\GGm$, &
{h)} & $\mu_3 \rtimes \Sigma_2$ \\
\end{tabular}}
\hfill
\vspace{8 pt}

\noindent
and we will compute the classes of their classifying stacks to

\vspace{8 pt}
\hskip -4pt
\begin{tabular}{llllllll}
{a)} & $\LL^{-3}(\LL + 1)^{-1}(\LL - 1)^{-2}$ &
{b)} & $\LL^{-2}(\LL - 1)^{-2}$ &
{c)} & $\LL^{-2}(\LL - 1)^{-1}$ \\
{d)} & $\LL^3(\LL^2-1)^{-1}(\LL^3-1)^{-1}$ &
{e)} & $\LL(\LL^2 - 1)^{-1}$ &
{f)} & $\LL^{-1}(\LL - 1)^{-1}$ \\
{g)} & $(\LL - 1)^{-1}$ &
{h)} & $1$. \\
\end{tabular}
\vspace{8 pt}

\noindent
Among these classes, all but the class in case {\em c} have already been
treated.
Indeed, the groups in the cases {\em a}, {\em b}, {\em f} and {\em g} are special,
so the classes of their classifying stacks are inverses to the classes of the groups
themselves. The classes in the cases {\em e} and {\em d} were computed in the previous
section and {\em h} was given by Proposition~\ref{gro-class-u2}.

In case {\em c}, we have the group $\GGa^2 \rtimes G$, where $G$ is the subgroup
$\GGm \rtimes \Sigma_3$ of $\GL_2$ generated by its centre and the embedding of
$\Sigma_3$ induced by its irreducible 2-dimensional representation. The
inclusions $G \hookrightarrow \GL_2$ and
$\GGa^2 \rtimes G \hookrightarrow \GGa^2 \rtimes \GL_2$ both give rise to the
same quotient space. Since both groups on the right
hand side of these arrows are special, we get the relation
$$
\{\BB G\}\{\GL_2\} = \{\BB (\GGa^2\rtimes G)\}\{\GGa^2 \rtimes \GL_2\}  
$$
by Proposition~\ref{gro-exact}. This reduces the problem of computing
$\{\BB (\GGa^2\rtimes G)\}$ to computing $\{\BB G\}$. But $G$ is isomorphic
to $\GG_m\times\Sigma_3$, so $\{\BB G\} = (\LL-1)^{-1}$. Hence, the class of
$\BB (\GGa^2\rtimes G)$ is $\LL^{-2}(\LL - 1)^{-1}$.
\section{The gerbe of plane smooth cubics} 
\label{mod}
Recall that $\Hns$ denotes the space of smooth degree 3 hypersurfaces in
$\PP^2$. In the last section, we saw how the class of $\BB \PGL_3$ was related to the
class of the stack $[\Hns/\PGL_3]$. We shall now study the stack quotient
$[\Hns/\PGL_3]$ more closely. It may be worth noting that in
this section we will not need any restrictions on the base we are working
over. The results hold over $\Spec \ZZ$. 

Since all degree 3 hypersurfaces in $\PP^2$ are smooth genus~1 curves,
it seems natural to assume that $[\Hns/\PGL_3]$ is somehow related to the moduli
stack $\MC{1}{1}$ of elliptic curves.
The main result of this section is that $[\Hns/\PGL_3]$ is equivalent to the
neutral gerbe $\BB_{\MC{1}{1}} \UC[3]$ over $\MC{1}{1}$ associated to the
3-torsion subgroup $\UC[3]$ of the universal elliptic curve $\UC$.
We do this by first establishing an equivalence to the moduli
stack $\pMC{1}{3}$ of genus~1 curves polarised in degree~3,
which we define in the next section.

\subsection{Moduli of polarised genus 1 curves}
\label{mod-pol}
Consider a smooth genus~1 curve $C \to S$ over a scheme. Recall that we have
an exact sequence $1 \to \sPic^0_{C/S} \to \sPic_{C/S} \to \ZZ \to 1$, where
$\sPic_{C/S}$ is the Picard sheaf and the map to $\ZZ$ is the degree map. The
group $\sPic_{C/S}$ may therefore be written as a disjoint union of sheaves
$\sPic^d_{C/S}$ which are pre-images of the integers $d$ in $\ZZ$. By a
\term{polarisation} of $C$ in degree~$d$, we mean a global section of the
sheaf $\sPic^d_{C/S}$. Since for any morphism $S'\to S$ there is a natural
identification of $\sPic_{C/S} \times_S S'$ with $\sPic_{C_{S'}/S'}$, we may
pull back polarisations on $C \to S$ to $C_{S'} \to S'$.
This allows us to define the fibred category $\pMC{1}{d}$ of genus one curves
polarised in degree $d$. The objects are genus~1 curves together with degree~$d$
polarisations, and the morphisms are cartesian squares respecting these
polarisations.
That $\pMC{1}{d}$ is a stack follows from the sheaf property of $\sPic^d_{C/S}$.

We want to establish an equivalence between the stack quotient $[\Hns/\PGL_3]$ and
$\pMC{1}{3}$. First we give an explicit description of the pre-stack quotient
as a category fibred in groupoids over the category of schemes.
Its object are the same as the objects of $\Hns$, i.e., smooth genus~1
curves embedded in $\PP^2_T$ over some scheme $T$. Now let $f\colon
T' \to T$ be a morphism of schemes and let $\iota'\colon C' \hookrightarrow \PP^2_{T'}$
and $\iota\colon C \hookrightarrow \PP^2_{T}$ be objects over $T'$ and $T$
respectively. A morphism from $\iota'$ to $\iota$ over $f$ is then given by a pair $(\sigma, \alpha)$,
where $\alpha$ is an automorphism of $\PP^2_{T'}$ and $\sigma\colon C' \to C$ is a morphism
such that the diagram
$$
\begin{tikzpicture}[description/.style={fill=white,inner sep=2pt}]
\matrix (m) [matrix of math nodes, row sep=3em,
column sep=2.5em, text height=1.5ex, text depth=0.25ex]
{
C' & \PP^2_{T'} \\
C & \PP^2_T \\
};
\path[->,font=\scriptsize]
(m-1-1)
edge node[auto] {$\sigma$} (m-2-1)
(m-1-2)
edge node[auto] {$\PP^2(f)$} (m-2-2);
\path[right hook->,font=\scriptsize]
(m-1-1)
edge node[auto] {$\alpha \circ \iota'$} (m-1-2)
(m-2-1)
edge node[auto] {$\iota$}  (m-2-2);
\end{tikzpicture}
$$
is cartesian.
Now we define a 1-morphism $f\colon [\Hns/\PGL_3] \to \pMC{1}{3}$ of stacks. By the universal
property of stackification, it is enough to define it on the pre-stack quotient, which we
denote by $[\Hns/\PGL_3]^{\mathrm{pre}}$.
It takes objects $\iota\colon  C \hookrightarrow \mathbb{P}^2_T$ to pairs
$(C \to T, [\iota^*\mathcal{O}(1)])$ and morphisms $(\sigma, \alpha)$ to $\sigma$.
Note that $f$ is well-defined on objects since smooth degree~3 hypersurfaces of
$\PP^2$ are smooth genus~1 curves and well-defined on morphisms since the
automorphism $\alpha$ does not affect the isomorphism class of the pulled back
line bundle.

\begin{prop}
\label{mod-bs-pol}
The 1-morphism $f\colon [\Hns/\PGL_3] \to \pMC{1}{3}$ defined in the
paragraph above is an equivalence of stacks.
\end{prop}

\begin{proof}
Let $\iota\colon  C \to \PP_T^2$ be an object of the
pre-stack quotient $[\Hns/\PGL_3]^\mathrm{pre}$ over a scheme $T$, and denote
the structure maps to $T$ by $q\colon C \to T$ and $p\colon \PP^2_T \to T$ respectively.
We also use the shorthand notation $\sheaf{L}$ for the invertible sheaf $\iota^*\sheaf{O}(1)$.
In order to prove that $f$ is fully faithful, it is enough to prove that it induces an
isomorphism between the automorphism group of $\iota\colon  C \to \PP_T^2$ in $[\Hns/\PGL_3]$ and
the automorphism group of $(q\colon C \to T, [\sheaf{L}])$ in $\pMC{1}{3}$.

First we will prove that the $\sheaf{O}_T$-module homomorphism
$p_*\sheaf{O}(1) \to q_*\sheaf{L}$ corresponding to the embedding as described
in \cite[§4.2]{egaII} is an isomorphism.
Since this may be verified locally, we may assume that we have a short exact
sequence of quasi-coherent $\sheaf{O}_{\PP^2}$-modules
\begin{equation}
\label{mod-hyper}
0 \to \sheaf{O}(-3) \to \sheaf{O} \to \iota_* \sheaf{O}_C \to 0.
 \end{equation}
Tensoring with the fundamental sheaf $\sheaf{O}(1)$ and using the projection formula on
the last term gives a new short exact sequence
$$
0 \to \sheaf{O}(-2) \to \sheaf{O}(1) \to \iota_* \sheaf{L} \to 0.
$$
Pushing this forward to $T$ gives rise to the exact sequence
$$
0 \to p_*\sheaf{O}(-2) \to p_*\sheaf{O}(1) \to q_* \sheaf{L} \to
R^1p_*\sheaf{O}(-2). $$
The map in the middle is the canonical map mentioned above, and it is an isomorphism
since both the first and last terms vanish \cite[Thm.~III.5.1]{hartshorne1977}. This allows us to assume
that $\PP^2_T = \PP(q_*\sheaf{L})$ and that the embedding $\iota$ corresponds to the
canonical map $\varepsilon\colon q^*q_*\sheaf{L} \to \sheaf{L}$.

{\em The functor $f$ is faithful.} To prove this, it is enough to show that for any
automorphism of $\iota\colon  C \to \PP(q_*\sheaf{L})$ of the form $(\id_C, \alpha)$,
the $\PP(q_*\sheaf{L})$-automorphism $\alpha$ is the identity. This may be
verified locally. Hence we may assume that the automorphism $\alpha$ is of the form $\PP(\beta)$, where
$\beta$ is an $\sheaf{O}_T$-module automorphism of $q_*\sheaf{L}$. The criterion that
$\alpha$ fixes the embedding $\iota$ is that there exists an $\sheaf{O}_C$-module
automorphism $\gamma$ of $\sheaf{L}$ such that the diagram
\vspace{5 pt} \\
\hspace*{\fill}
\begin{tikzpicture}[node distance=2cm, auto]
\node (A) {$q^*q_*\sheaf{L}$};
\node (B) [right of=A] {$\sheaf{L}$};
\node (C) [below of=A] {$q^*q_*\sheaf{L}$};
\node (D) [below of=B] {$\sheaf{L}$};

\draw[->] (A) -- (B);
\draw[->] (C) -- (D);
\draw[->] (A) to node [swap] {$q^*\beta$} (C);
\draw[->] (B) to node {$\gamma$} (D);
\end{tikzpicture}
\hspace{\fill}
\vspace{5 pt} \\
commutes. By using the adjunction property of the pair $(q^*, q_*)$, we see that $\beta$ must be
of the form $q_*\gamma$. The automorphism $\gamma$ may be viewed as a global section of
$\sheaf{O}_C^\times$. If we apply $p_*$ to the exact sequence (\ref{mod-hyper}), we get the
exact sequence
$$
0 \to p_*\sheaf{O}(-3) \to p_*\sheaf{O} \to q_* \sheaf{O}_C \to
R^1p_*\sheaf{O}(-3). $$
Since both $p_*\sheaf{O}(-3)$ and $R^1p_*\sheaf{O}(-3)$ vanish, we see that
$q_* \sheaf{O}_C \cong p_*\sheaf{O}$, with the latter sheaf being isomorphic to $\sheaf{O}_T$.
Hence $q_*\gamma$ is a global section of $\sheaf{O}_T^\times$. It follows that the automorphism
$\alpha$ is the identity.

{\em The functor $f$ is full.} To prove this, we need to verify that the map on automorphisms
is surjective. Let $\sigma$ be a $T$-automorphism of $C$ such that
$[\sigma^*\sheaf{L}] = [\sheaf{L}]$ in $\sPic_{C/T}(T)$. It is enough to show that
$\sigma$ locally is given by an automorphism of $\PP(q_*\sheaf{L})$, so we may assume that
$\sigma^*\sheaf{L} \simeq \sheaf{L}$. The new embedding $\iota \circ \sigma$
then corresponds to the automorphism $\alpha\colon q_*\sheaf{L} \to
q_*\sheaf{L}$ given by $s \mapsto \sigma^*(s)$.
It follows that $\PP(\alpha)\colon \PP(q_*\sheaf{L}) \to \PP(q_*\sheaf{L})$ is
our sought automorphism.

{\em The functor $f$ is essentially surjective.} This may also be checked fppf-locally.
Hence, given an object $(q\colon C \to T, \lambda)$ of $\pMC{1}{3}$, we may assume that $\lambda$
comes from a line bundle $\sheaf{L}$ of degree $3$ on $C$.
The push forward $q_*\sheaf{L}$ is locally free of rank 3, and we get an
embedding of $C$ into the projective bundle $\PP(q_*\sheaf{L})$. This is classic
in the case where the base is a field, and follows from a cohomology and base
change argument in the relative case. By extending the base further if
necessary, we may assume that this bundle is $\PP^2_T$, so our object
$(q\colon C \to T, \lambda)$ comes from an object of $[\Hns/\PGL_3]^\text{pre}$.
\end{proof}

\subsection{An interlude on torsors}
\label{mod-tor}
If $A$ is a sheaf of abelian groups, the contraction product of two
$A$-torsors gives a new torsor. In particular, we may form contraction powers
$T^i$ of a torsor $T$.
We may also take the inverse $T^{-1} := \sIsom_A(T, A)$, giving the set
$\HH^1(S, A)$ of isomorphism classes of torsors a group structure and
the stack $\BB A$ the structure of a Picard stack. In the abelian case, the
stack $\BB A$ is isomorphic to the stack $\stExt{1}(\ZZ, A)$ of extensions of
$\ZZ$ by $A$. The $S$-points of this stack are short exact sequences
$$
0 \to A_S \to E \to \ZZ_S \to 0
$$
of sheaves over $S$ with morphisms of such sequences being isomorphisms of
complexes. Given such a short exact sequence, we get a corresponding torsor by
taking the elements in $E$ mapping to $1$. Conversely, given a torsor
$T_S$ over $S$, we may form the exact sequence
$$
0 \to A_S \to \coprod_{i \in \ZZ} T_S^i \to \ZZ_S \to 0.
$$
A more detailed account on this equivalence is given in  \cite[Exposé~VII]{SGA7I}.
It is interesting to note that the decategorification of this functor induces the
usual isomorphism $\HH^1(S, A) \to \Ext^1(\ZZ, A)$.

In order to describe $\pMC{1}{3}$ as a classifying stack, we would like to
reinterpret polarisations in terms of torsors. It turns out that much of this
may be worked out in the general theory for torsors for abelian sheaves over an
arbitrary site $\stack{C}$. Hence we make a short interlude, working in this
generality.

Let $A$ be a fixed sheaf of abelian groups on $\stack{C}$. Given an $A$-torsor
$T$ and a positive integer $n$, we have a map $n_T\colon T \to T^n$ taking a
local section $t$ of $T$ to its $n$-fold contraction power $(t, \ldots , t)$.
In particular, $n_A\colon A \to A$ is the map taking a generalised point $a$ to
its $n$-fold product $a^n$ using the group law. The kernel of this map is the $n$-torsion
subgroup of $A$, which we denote by $A[n]$. If $A$ is an $n$-torsion group,
there is a canonical identification of $T^n$ with $A$ for each $A$-torsor $T$.

\begin{lemma}
\label{mod-can}
Let $A$ be a sheaf of abelian groups on a site $\stack{C}$ and $T$ an $A$-torsor.
If $A$ is an $n$-torsion group, then the torsor $T^n$ has a canonical global
section $\kappa$.
\end{lemma}
\begin{proof}
Fix an object $S \in \stack{C}$ such that $T(S)$ is non-empty and let $x, y \in
T(S)$. Then $y = a\cdot x$ for some group element $a \in A(S)$. We have 
$y^n = (a\cdot x)^n = a^n\cdot x^n = x^n$, since $A$ is an $n$-torsion group.
It follows that $T^n$ has a canonical $S$-point $\kappa_S = x^n$.
Taking a covering $S_i$ such that $T(S_i)$ has sections, the canonical local
sections $\kappa_{S_i}$ glue together to the global section $\kappa$.
\end{proof}

We define the category $\BB_n A$, fibred over $\stack{C}$, as the category of pairs
$(T\to S, \lambda\colon S \to T^n)$, where $T\to S$ is an $A$-torsor over some
object $S$ in $\stack{C}$ and $\lambda$ is a global section of $T^n$. Morphisms
are pullbacks of sheaves respecting the global sections. Recall that the
inclusion $A[n] \to A$ induces a morphism $\BB A[n] \to \BB A$ taking an
$A[n]$-torsor $T$ to the  $A$-torsor $T' = A \contr{A[n]}T$. The canonical
global section $\kappa$ of $T^n$ allows us to define a canonical global section
$(1, \kappa)$ of $(T')^n \cong A \contr{A[n]} T^n$. Hence we get a natural map
$$
\BB A[n] \to \BB_n A
$$
through which $\BB A[n] \to \BB A$ factors. This is not an equivalence in
general, but we have the following result.

\begin{prop}
\label{mod-equiv}
Let $A$ be a sheaf of abelian groups, and assume that the map
$n_A\colon A \to A$ is surjective. Then the natural
map $\BB A[n] \to \BB_n A$ is an equivalence of stacks.
\end{prop}
\begin{proof}
We prove the equivalence by constructing a 2-inverse explicitly.
Given an object $(T\to S, \lambda)$, we may define the subsheaf
$T_\lambda \subset T$ over $S$ as the pullback of
$n_T\colon T \to T^n$ along the map $\lambda\colon S \to T^n$. On
$S'$-points, this may be described as
$$
T_\lambda(S') := \{x \in T(S') \mid x^n = \lambda \text{ in } T^n(S')\}.
$$
From this description it is straightforward to verify that the
$A$-action on $T$ restricts to a well-defined $A[n]$-action on
$T_\lambda$. This is free and transitive, making $T_\lambda$ a pseudo-torsor
for $A[n]$. Locally, the morphism $n_T$ is just $n_A$, so $n_T$ is surjective.
Hence the same holds for the structure map $T_\lambda \to S$, which proves that
$T_\lambda$ actually is a torsor.

Given objects $(T, \lambda)$ and $T'$ in $\BB_n A$ and $\BB A[n]$
respectively, we have natural maps
$$
\eta_{(T, \lambda)}\colon A\contr{A[n]}T_\lambda \to T, \qquad
\varepsilon_{T'} \colon T' \to (A\contr{A[n]}T')_{(1, \kappa)}
$$
given on generalised points by $(a, t) \mapsto at$ and
$t \mapsto (1, t)$ respectively. The reader may verify that these are
isomorphisms in the categories $\BB_n A$ and $\BB A[n]$ respectively.
\end{proof}

\subsection{The stack of polarised genus~1 curves as a gerbe}
\label{mod-ger}
Now we apply the results from the previous subsection to our situation
with the stack $\pMC{1}{n}$ to show that it is a gerbe over $\MC{1}{1}$.
However, we cannot use the result directly, since the base $\MC{1}{1}$ is
a stack rather than a scheme. Proposition~\ref{mod-equiv} could be generalised
to this situation, but we shall instead just give the explicit description in
this special case. Since it should be easy to fill in the details, we shall
allow ourselves to be somewhat sketchy.

The fibred category $\BB_{\MC{1}{1}}\UC[n]$ over schemes has pairs
$$
(E\to S, T \to S),
$$
as objects, where $E \to S$ is an elliptic curve and $T \to S$ is an
$E[n]$-torsor. We will use the rest of the section to prove the following
proposition:
\begin{prop}
\label{mod-pol-class}
The stack $\pMC{1}{n}$ is equivalent to $\BB_{\MC{1}{1}}\UC[n]$.
\end{prop} 
Consider the fibred category $(\BB_n)_{\MC{1}{1}}\UC$ over schemes
whose objects are triples
$$
(E\to S, T \to S, \lambda\colon S\to T^n),
$$
where $E \to S$ is an elliptic curve and $T \to S$ is an $E$-torsor.
Now let $(C \to S, \lambda\colon S \to \sPic^n_{C/S})$ be an object
of $\pMC{1}{n}$. Since the Picard sheaf $\sPic_{C/S}$ is an extension
of $\ZZ_S$ by $\sPic^0_{C/S}$, the component $\sPic^1_{C/S}$ is a
$\sPic^0_{C/S}$-torsor and $\sPic^n_{C/S}$ is canonically isomorphic
to its $n$-th contraction power. The group $\sPic^0_{C/S}$ is an
elliptic curve, being the Jacobian of a genus~1 curve. Hence we get
a well-defined 1-morphism $\pMC{1}{n} \to (\BB_n)_{\MC{1}{1}}\UC$
over $\MC{1}{1}$ taking the object to
$$
(\sPic^0_{C/S},\sPic^1_{C/S}, \lambda\colon S \to \sPic^n_{C/S}).
$$
Note that since $C \to S$ is a smooth genus~1 curve, there is a canonical
isomorphism $C \to \sPic^1_{C/S}$, so this 1-morphism has an
obvious 2-inverse.

Now we consider the functor
$f\colon \BB_{\MC{1}{1}}\UC[n] \to (\BB_n)_{\MC{1}{1}}\UC$.
This is defined analogously with the equivalence in the previous section
by taking $(E \to S, T\to S)$ to $(E \to S, E\contr{E[n]}T, (1, \kappa))$.
For an arbitrary scheme $S$ and a morphism $S \to \MC{1}{1}$, corresponding to
an elliptic curve $E \to S$, the functor above pulls back to the functor
$f_S\colon \BB_S E[n] \to (\BB_n)_{S}E$. Since $n_E: E \to E$ is an
isogeny, and in particular surjective on the underlying sheaves, we are now in the
situation where we can apply Proposition~\ref{mod-equiv}. Therefore $f_S$,
and hence also $f$, is an equivalence, and we are done.

\begin{remark}
Let $k$ be a field and $E/k$ an elliptic curve.
Then the moduli interpretations of the stacks
$\BB_{\MC{1}{1}}\UC[n]$, $[\Hns/\PGL_3]$ and~$\pMC{1}{n}$
respectively give the following descriptions of equivalence
classes of $k$-points mapping to $E\colon \Spec k \to \MC{1}{1}$:
\begin{enumerate}
\item[a)]
The set of isomorphism classes of $E[3]$-torsors over $k$.
\item[b)]
The set of equivalence classes of embeddings $C \hookrightarrow S$,
where $C$ is an $E$-torsor and $S$ a 2-dimensional Brauer--Severi variety over~$k$.
\item[c)]
The set of equivalence classes of pairs $(C, \lambda)$ where $C$ is an
$E$-torsor and $\lambda$ is a divisor class on $C$ of degree 3.
\end{enumerate}
Thus we recover three different more or less well-known
interpretations of the cohomology group $H^1(k, E[3])$.
For further discussion of these and other interpretations,  
see~\cite[§1]{cfoss2008}).
\end{remark}

\section{Groups of 3-torsion points of elliptic curves}
\label{com-smo}
Fix an arbitrary prime $\ell$ and let $S$ be a scheme over a field $k$. By a
rank $n$ \term{local system} for $\GF{\ell}$ over a scheme $S$, we mean a sheaf
$V$ which is locally isomorphic to an $n$-dimensional $\GF{\ell}$-vector space
considered as a constant sheaf. A \term{symplectic local system} is a
pair $(V, \omega)$, where $V$ is a local system and
$\omega\colon V \times V \to \GF{\ell}$ is a non-degenerate symplectic form.

Such local systems arise in the study of $\ell$-torsion of elliptic curves
$E/S$.
In general, the $\ell$-torsion subgroup $E[\ell]$ of $E$ is a degree $\ell^2$
finite flat group scheme over $S$. If $\ell$ is invertible in $k$, then $E[\ell]$
is étale and becomes a local system for $\GF{\ell}$. The Weil pairing
$\omega\colon E[\ell]\times E[\ell] \to \mu_\ell$ is a twisted symplectic form on
$E[\ell]$, making $E[\ell]$ self-dual. If $k$ contains all $\ell$-th roots of unity,
then $\mu_\ell = \GF{\ell}$ and $(E[\ell], \omega)$ is a symplectic local system of
rank 2. We compute the class of $\BB E[\ell]$ in the case $\ell = 3$ under these
hypotheses on $k$.

\begin{prop}
\label{com-tors}
Let $S$ be a scheme and let $(V, \omega)$ be a rank 2 symplectic local system for
$\GF{3}$ over $S$. Then the class $\{\BB_S V^\vee\} = 1$ in $\KK_0(\cStack_S)$.
\end{prop}
\begin{proof}
Since the question is Zariski local on $S$, we may assume that $S$ is connected.
We first assume that $\ell$ is any prime.

Let $\Gamma = \pi_1(S, \overline{\xi})$ denote the étale fundamental group of $S$
with respect to some geometric point $\overline{\xi} \in S$.
Then the pair $(V, \omega)$ corresponds to a pair consisting of a $2$-dimensional
$\Gamma$-representation over $\GF{\ell}$ and a $\Gamma$-invariant symplectic form.
By abuse of notation, we denote this pair by $(V, \omega)$ as well.

Let $V_0 \subset V$ denote the $\Gamma$-invariant subset where the origin in
$V$ has been removed. Then the free abelian group $\ZZ[V_0]$ on $V_0$ has a natural
structure of $\Gamma$-module and we have a natural $\Gamma$-equivariant map
$\ZZ[V_0] \to V$, which takes a formal sum of elements of $V_0$ to an actual sum in $V$.
This gives rise to an exact sequence
$$
0 \to K \to \ZZ[V_0] \to V \to 0
$$
of $\Gamma$-modules.

Denote the set of lines through the origin in $V$ by $\PP(V)$.
Then we have a surjection $V_0 \to \PP(V)$ of $\Gamma$-sets inducing a 
surjection $\ZZ[V_0] \to \ZZ[\PP(V)]$ of $\Gamma$-modules.

If we restrict to the case when $\ell$ is odd, then the map $K \to \ZZ[\PP(V)]$
given by composition is also a surjection. Indeed, each standard basis element
$(\mu\hcolon\lambda)$ of $\ZZ[\PP(V)]$ lifts to $2(\mu,\lambda)-(2\mu,2\lambda)$ in $K$.

The symplectic form $\omega$ allows us to define an endomorphism $\varphi$ on $\ZZ[V_0]$ by
$$
[v] \mapsto \sum_{\omega(v, u) = 1} [u],\qquad v, u \in V_0.
$$
This is $\Gamma$-equivariant since $\omega$ is $\Gamma$-invariant. The image of $\varphi$
lies in $K$. This can be seen by choosing $v'$ such that $\omega(v, v') = 1$ and letting $W$
be the subspace of vectors $u$ such that $\omega(v, u) = 0$. Then $v$ maps to
$\# W\cdot v' + \sum_{u \in W} u$ in $V$, which indeed is zero.

The endomorphism $\varphi$ descends to a corresponding endomorphism $\varphi'$
on $\ZZ[\PP(V)]$ given by
$$
[P] \mapsto \sum_{\omega(P, Q) \neq 0} [Q],\qquad P, Q \in \PP(V).
$$
Since $\PP(V)$ has $\ell + 1$ points, this endomorphism is described by
an $\ell + 1$ by $\ell + 1$ matrix with respect to the standard basis.
All the elements of this matrix are one, except for the elements on the main
diagonal which are zero. Since such a matrix has determinant $-\ell$, it follows
that $\varphi'$ is injective with cokernel of order $\ell$.

Now assume that $\ell = 3$. A straightforward computation gives $\det \varphi =
- 3^3$. In particular, the map $\varphi$ is injective. We summarise the
situation in the following diagram with exact rows and columns:
$$
\begin{tikzpicture}[description/.style={fill=white,inner sep=2pt}] \matrix (m) [matrix of math nodes, row sep=2.5em,
column sep=2em, text height=1.5ex, text depth=0.25ex]
{
& 0 & 0 & 0\\
0 & K'' & K' & A & 0\\
0 & \ZZ[V_0] & K & B & 0 \\
0 & \ZZ[\PP(V)] & \ZZ[\PP(V)] & \ZZ/\ell\ZZ & 0 \\
& 0 & 0 & 0\\
};
\path[->,font=\scriptsize]
(m-1-2)
edge (m-2-2)
(m-1-3)
edge (m-2-3)
(m-1-4)
edge (m-2-4)

(m-2-1)
edge (m-2-2)
(m-2-2)
edge (m-2-3)
edge (m-3-2)
(m-2-3)
edge (m-2-4)
edge (m-3-3)
(m-2-4)
edge (m-2-5)
edge (m-3-4)

(m-3-1)
edge (m-3-2)
(m-3-2)
edge node[auto] {$ \varphi $} (m-3-3)
edge (m-4-2)
(m-3-3)
edge (m-3-4)
edge (m-4-3)
(m-3-4)
edge (m-4-4)
edge (m-3-5)

(m-4-1)
edge (m-4-2)
(m-4-2)
edge node[auto] {$ \varphi' $} (m-4-3)
edge (m-5-2)
(m-4-3)
edge (m-4-4)
edge (m-5-3)
(m-4-4)
edge (m-4-5)
edge (m-5-4);
\end{tikzpicture}
$$
\noindent
Since $K$ has index $\ell^2$ in $\ZZ[V_0]$, it follows that $B$ has order
$\ell$, which forces $A = 0$ by exactness of the last column.

Next we take the Cartier dual of the diagram. The maps $\ZZ[V_0]^\vee \to K''^\vee$ and
$K^\vee \to K'^\vee$ are both $\ZZ[\PP(V)]^\vee$-torsors. Since
the torus $\ZZ[\PP(V)]^\vee$ is quasi-split, and therefore special, we get
the equalities $\{\ZZ[\PP(V)]^\vee\}\{K''^\vee\} = \{\ZZ[V_0]^\vee\}$ and
$\{\ZZ[\PP(V)]^\vee\}\{K'^\vee\} = \{K^\vee\}$. Since we have seen that $K'
\simeq K''$, it follows that $\{\ZZ[V_0]^\vee\} = \{K^\vee\}$.

But $\ZZ[V_0]^\vee$ is also a quasi-split torus. Hence the result follows by applying
Proposition~\ref{gro-exact} to the exact sequence
$$
0 \to V^\vee \to \ZZ[V_0]^\vee \to K^\vee \to 0.
$$
\end{proof}

\begin{remark}
The proof fails for $\ell > 3$ since in this case $A$ in the above diagram does not vanish.
Experiments suggest that the determinant of $\varphi$ is given by 
$\det \varphi = (-1)^{\frac{\ell-1}{2}}\ell^{\binom{\ell}{2}}$.
\end{remark}

\begin{corollary}
\label{com-gerbe-class}
Let $k$ be a field of characteristic not equal to $3$ containing all third
roots of unity. Then $\{\BB_S E[3]\} = \{S\}$ in $\KK_0(\cStack_k)$ for all
elliptic curves $E/S$ with $S \in \cStack_k$. In particular
$\{\pMC{1}{3}\} = \{\MC{1}{1}\}$.
\end{corollary}
\begin{proof}
The case when $S$ is a scheme follows directly from Proposition~\ref{com-tors}
and the remarks preceding it.
The case with $S$ an arbitrary stack follows by applying Proposition~\ref{gro-st-loc} with
$C = 1$.
The identity $\{\pMC{1}{3}\} = \{\MC{1}{1}\}$ follows from the special case
$S = \MC{1}{1}$ together with the identification $\pMC{1}{3} \simeq
\BB_{\MC{1}{1}} \UC[3]$ from Proposition~\ref{mod-pol-class}.
\end{proof}

\begin{prop}
\label{com-m11-class}
Let $k$ be a field in which $6$ is invertible.
Then $\{\MC{1}{1}\} = \LL$ in $\KK_0(\cStack_k)$.
\end{prop}
\begin{proof}
For ease of notation, we denote $\MC{1}{1}$ by $\stack{M}$. There is a map
$j\colon\stack{M} \to \AA^1 = \Spec k[t]$ to the coarse moduli space induced by
the classical $j$-invariant. Consider the closed points $\{0\}$ and~$\{1728\}$ in
$\AA^1$, and denote their complement by $U$. This induces a stratification of
$\stack{M}$ into the closed substacks $(\stack{M})_0$ and $(\stack{M})_{1728}$
and the open complement $\stack{M}_U$.

The stack $\stack{M}_U$ is equivalent to $\BB_U \Sigma_2$ over $U$.
Indeed, the inertia of $\stack{M}_U \to U$ is the automorphism group
of the universal elliptic curve $\UC_U \to \stack{M}_U$, which is
$\Sigma_2$ since we removed the curves with $j$-invariants 0 or 1728.
In particular, the inertia stack is faithfully flat of finite presentation
over $\stack{M}_U$, so $\stack{M}_U \to U$ is a gerbe. Moreover,
we see that it is the neutral gerbe since $\stack{M}_U \to U$ has a section.
This section is induced by the elliptic curve $E$ defined by the equation
$$
y^2z + xyz = x^3 - \frac{36}{t - 1728}xz^2 - \frac{1}{t - 1728}z^3
$$
over $U$. It is a straightforward computation to check that $E \to U$ is an
elliptic curve whose fibres $E_t$ have $j$-invariant $t$ over closed
points $t \in U$.

It is of course easy to construct elliptic curves with $j$-invariants $0$ and
$1728$ over $k$, so both the stacks $(\stack{M})_0$ and $(\stack{M})_{1728}$ are
neutral gerbes over $k$. Since we assume that 6 is invertible in the base field,
the automorphism groups of elliptic curves with $j$-invariants $0$ and~$1728$
are $\mu_6$ and~$\mu_4$ respectively \cite[3.4]{husemoller2004}. It follows
that $(\stack{M})_0 \simeq \BB \mu_6$ and $(\stack{M})_{1728} \simeq \BB \mu_4$.

Now it follows by Proposition~\ref{gro-units} and the scissors relations that
the class of $\stack{M}$ equals
$$
\{\BB_U \Sigma_2\} + \{\BB \mu_6\} + \{\BB \mu_4\} = \LL-2 + 1+ 1 = \LL
$$
in $\KK_0(\cStack_k)$.
\end{proof}

As a corollary we get the final piece of information we need in order to prove
Theorem~\ref{mainBPGL}.

\begin{corollary}
Let $k$ be a field of characteristic in which $6$ is invertible containing all third
roots of unity. Then $[\Hns/\PGL_3] = \LL$ in $\KK_0(\cStack_k)$.
\end{corollary}
\begin{proof}
This follows from Proposition~\ref{com-m11-class} combined with
Corollary~\ref{com-gerbe-class} and the equivalence
$[\Hns^3/\PGL_3] \simeq \pMC{1}{3}$ established in
Proposition~\ref{mod-bs-pol}.
\end{proof}
\noindent
The corollary, together with the classes of the strata corresponding to
singular curves computed in Section~\ref{com-sin}, shows that the class of
$\BB \PGL_3$ lies in $\Phi_\LL^{-1}\ZZ[\LL]$, and hence must coincide with
the expected class, which is the inverse of $\{\PGL_3\}$. Of course, since we
actually have computed the classes of all strata of $[H^3/\PGL_3]$ explicitly,
we can also get the result by just taking the sum of the classes.

\appendix
\section{Singular plane cubics and stabilisers}
\label{sin}
Throughout the appendix, we let $k$ be a field in which $6$ is invertible.
It is a classical result that there exist eight singular cubic curves
in $\PP^2_k$ up to projective equivalence (see, for instance,
\cite[I.7]{kraft1984}).
These correspond to orbits in the space of singular cubics in $\PP^2$ under the
natural action of $\PGL_3$ by change of coordinates.
In this appendix, we will determine the stabiliser groups corresponding to these
orbits up to isomorphism.
The result is described in the table below.

\vspace{8 pt}
\begin{center}
\begin{tabular}{c|l|c|c}
& Description & Standard Form & Stabiliser \\
\hline
{a)} & Triple line & $x^3$ & $\GGa^2 \rtimes \GL_2$ \\
{b)} & Double and single line & $x^2y$ &  $\GGa^2 \rtimes \GGm^2$ \\
{c)} & Three lines through a point & $x^2y + xy^2$ &  $\GGa^2 \rtimes G$ \\
{d)} & Three general lines & $xyz$ &  $\PMon_3$ \\
{e)} & Conic and general line 
& $xyz + z^3$ & $\GGm \rtimes \Sigma_2$ \\ 
{f)} & Conic and tangent line & 
$y^2z+x^2y$ & $\GGa \rtimes \GGm$ \\ 
{g)} & Cuspidal cubic & $x^2z + y^3$ & $\GGm$ \\ 
{h)} & Nodal cubic & $xyz + x^3 + y^3$ & $\mu_3 \rtimes \Sigma_2$ \\
\end{tabular}
\end{center}
\vspace{8 pt}
\begin{remark}
We avoid characteristic 2 or 3 since some of the stabiliser groups are
non-reduced in these cases. In characteristic 3, there will also be two
different orbits corresponding to cuspidal cubics. See \cite{af2000} for
a short discussion on this.
\end{remark}
\begin{remark}
The author recently learned that similar computations,
but with the group acting being $\GL_3$, has been done in~\cite{bf2004}.
\end{remark}

\noindent
The table lists the type of singular curve, the equation for a prototypical
curve, and the stabiliser group up to isomorphism. The symbol $\PMon_3$ denotes
the normaliser of the maximal torus in $\PGL_3$. Explicitly, this group may be
described as the group of monomial $3\times3$ matrices up to multiplication by a
scalar. The group denoted by $G$ is the subgroup of $\GL_2$ generated by the
scalar matrices and the embedding of $\Sigma_3$ in $\GL_2$ induced by the
2-dimensional irreducible representation.

In characteristic zero, the stabiliser groups are determined by the points of the underlying
topological space. In positive characteristic however, we must also account for the possibility
of the stabilisers not being reduced. Our assumption on the base field asserts that this situation
does not occur. This may be verified by determining the dimension of the Lie algebra for the
stabiliser. We will go through these arguments in detail for the first computations only
and leave the rest for the reader to verify.

When using coordinates in our arguments, we use the convention that $\PGL_3$ acts
by standard transformation of coordinates on $\PP^2$ from the left. This means that the
action on forms is dual and hence is a right action. We will frequently represent elements
in $\PGL_3$ as $3\times3$-matrices. When doing so, taking the quotient by the scalar matrices
is implicit. The corresponding convention applies when we discuss the Lie algebra of $\PGL_3$.

\subsection{Three Lines}
First we treat the case when the form defining the curve is a product of three linear forms.
There are four different configurations to consider.

\subsubsection*{(a) A triple line}
We choose our prototypical curve such that it is defined by the form $x^3$.
On points, this is the same as the stabiliser of the line $x = 0$. This consists of the
matrices $(a_{ij})$ such that $a_{12} = a_{13} = 0$. By normalising
the coordinates by setting $a_{11} = 1$, we get that this group is isomorphic
to $\GG_a^2 \rtimes \GL_2$.

Now let $A = I + \varepsilon(a_{ij})$ be a general element of the Lie algebra
of $\PGL_3$, i.e.\ a $\overline{k}[\varepsilon]$-point mapping to the identity.
Then $x^3 \cdot A$ is 
$$
x^3 + 3\varepsilon x^2(a_{11}x + a_{12}y + a_{13}z).
$$
This gives the condition $3a_{12} = 3a_{13} = 0$. Since we assume that $3$ is invertible,
we get $a_{12} = a_{13} = 0$. We conclude that both the Lie-algebra and the
group have the same dimension, so the stabiliser is smooth and therefore reduced. Therefore
$\PGL^3_{x^3} \simeq \GG_a^2 \rtimes \GL_2$.

\subsubsection*{(b) A double and a single line}
This time we choose $x^2y$ as our standard representative. An element $(a_{ij})$ 
of the stabiliser must preserve both the line $x = 0$ as well as the
line $y = 0$. This forces $a_{12} = a_{13} = a_{21} = a_{23} = 0$. By normalising
$a_{33} = 1$, we see that the reduced stabiliser is $\GG_a^2 \rtimes \GG_m^2$.

Now we consider a general element $A = I + \varepsilon(a_{ij})$ of the Lie-algebra in the
same way as in the previous case. Then we get that $x^2y\cdot A$ equals
$$
x^2y + \varepsilon(2xy(a_{11}x + a_{12}y + a_{13}z) + x^2(a_{21}x + a_{22}y + a_{23}z)).
$$
Since $2$ is invertible, this gives the conditions $a_{12} = a_{13} = a_{21} = a_{23} = 0$.
Again we see that the dimension is right, so we get $\PGL^3_{x^2y} \simeq \GG_a^2 \rtimes \GG_m^2$.
 
\subsubsection*{(c) Three lines intersecting at a single point}
Let $x^2y + xy^2$ be our standard form. An element $(a_{ij})$ of the stabiliser must preserve
the intersection point $(0\hcolon0\hcolon1)$ of the three lines. This forces $a_{13} = a_{23} = 0$.
By normalising $a_{33} = 1$, we see that the stabiliser is a subgroup of $\GGa^2 \rtimes \GL_2$.

Next we determine the stabiliser of our standard form under the action of the subgroup $\GL_2$.
This corresponds to the problem of finding the stabiliser of an unordered triple of distinct points in
$\PP^1$ under the standard action of $\GL_2$. Since the corresponding action of $\PGL_2$ is simply
3-transitive, the stabiliser is an extension of $\Sigma_3$ by $\GG_m$. One verifies that this is
the subgroup $G$ as described in the introduction of this section. Since the subgroup $\GGa^2$ clearly
stabilises our standard form, we get that the reduced stabiliser is $\GGa^2
\rtimes G$.

The corresponding calculation for the Lie-algebra for $\PGL_{x^2y + xy^2}^3$ as in the previous cases
leads to the relations $a_{11} = a_{22}$ and $a_{12} = a_{13} = a_{21} = a_{23} = 0$. This shows that
the dimension is right regardless of the characteristic, so $\PGL_{x^2y +
xy^2}^3 \simeq \GGa^2 \rtimes G$.

\subsubsection*{(d) Three lines in general position}
Choose $xyz$ as standard form. The stabiliser has to preserve the lines $x$, $y$ and $z$ up to
permutation. If we impose an ordering on the lines, the stabiliser consists of the diagonal
matrices. Since we may reorder the lines by using permutation matrices, the group $\PGL^3_{xyz}$
is generated by the diagonal and the permutation matrices. One verifies that also in this case the
stabiliser is smooth regardless of characteristic. Thus the stabiliser $\PGL^3_{xyz}$ is the
group $\PMon_3$ described in the introduction of the appendix.

\subsection{A smooth conic and a line}
There are two types of cubic curves consisting of a smooth conic
and a line. The line is either tangent to the cubic or intersects it at
two distinct points.

\subsubsection*{(e) Smooth conic and non-tangent line}
The intersection between the curves determine an unordered pair of
points $P1$, $P2$. The tangents to the conic at these points are distinct by
Bézout's Theorem. Hence they intersect in a third point $Q$.

Chose $xyz + z^3$ as the standard form defining our cubic. For this curve,
the points $P1$, $P2$ and $Q$ as defined above have coordinates
$(1\hcolon 0\hcolon 0)$, $(0\hcolon 1\hcolon 0)$ and $(0\hcolon 0\hcolon 1)$
respectively. Any element of the stabiliser $\PGL^3_{xyz + z^3}$ must preserve
these points, so the stabiliser is contained in the group generated by the
diagonal matrices and the permutation matrix switching the first two coordinates.

The subgroup of the diagonal matrices $\diag(a \hcolon b \hcolon c)$ stabilising
the form $xyz + z^3$ is defined by the equation $abc = c^3$. Hence it is
isomorphic to $\GG_m$, as seen by the parametrisation
$t \mapsto \diag(t \hcolon t^{-1} \hcolon 1)$. It follows that the stabiliser is
isomorphic to $\GGm \rtimes \Sigma_2$, where $\Sigma_2$ acts non-trivially on
$\GGm$.

\subsubsection*{(f) Smooth conic and tangent line}
In this case, we let $y^2z + x^2y$ be the standard form defining our curve.
The stabiliser $\PGL^3_{y^2z + x^2y}$ must preserve the intersection
point $(0 \hcolon 0 \hcolon 1)$ between the conic and the tangent line, as well
as the tangent line $y = 0$ itself. Hence it must be a subgroup of the group of
projective matrices of the following form:
$$
\left (
\begin{array}{ccc}
a_{11} & a_{12} & 0 \\
0 & a_{22} & 0 \\
a_{31} & a_{32} & a_{33} \\
\end{array}
\right )
$$
The additional requirement that it also should preserve the conic $yz + x^2$
gives the equations
$$
a_{11}^2 = a_{22}a_{33},\qquad
a_{12}^2 + a_{22}a_{32} = 0,\qquad
a_{22}a_{31} + 2a_{11}a_{12} = 0.
$$
This resulting subgroup is isomorphic to $\GGa \rtimes \GGm$, which may be seen
by using the parametrisation described below.
$$
\left (
\begin{array}{cc}
a & 0 \\
b & 1 \\
\end{array}
\right )
\mapsto
\left (
\begin{array}{ccc}
a & -ab & 0 \\
0 & a^2 & 0 \\
2b & -b^2 & 1 \\
\end{array}
\right )
$$

\subsection{An integral cubic}
There are two types of integral singular cubics, both having exactly one singularity.
The singularity is either a node or a cusp.

\subsubsection*{(g) Cuspidal cubic}
A cuspidal cubic has exactly one singularity and one inflection point.
We denote these points by $P1$ and $P2$ respectively. Consider the reduced line
associated to the tangent cone at $P1$ and the tangent line at $P2$.
As a consequence of Bézout's Theorem, these lines are distinct.
Hence they have a unique intersection point $P3$, which does
not lie on the line between $P1$ and $P2$.

We choose the standard cuspidal cubic $x^2z + y^3$. In this case, the coordinates of the
points $P1$, $P2$ and $P3$ as described above are $(0\hcolon 0\hcolon 1)$, 
$(1\hcolon 0\hcolon 0)$ and $(0\hcolon 1\hcolon 0)$ respectively. 
Since the stabiliser $\PGL^3_{x^2z + y^3}$
preserves these points, it must be a subgroup of the group of diagonal matrices.
Introducing coordinates $\diag(a \hcolon b \hcolon c)$ for these matrices, we see
that the stabiliser is the subgroup defined by the equations $a^2c = b^3$. This group
is isomorphic to $\GG_m$ via the parametrisation
$t \mapsto \diag(t^3 \hcolon t^2 \hcolon 1)$.

\subsubsection*{(h) Nodal cubic}
The standard nodal cubic $xyz + x^3 + y^3$ has the tangent cone $xy = 0$ at
the singularity. All its three inflection points are distinct and lie at the line $z = 0$ at infinity.
Hence the stabiliser group $\PGL^3_{xyz + x^3 + y^3}$ must preserve the forms $xy$ and $z$. It is
therefore a subgroup of the group generated by the diagonal matrices and the
permutation matrix exchanging the $x$- and $y$-coordinates. The diagonal
matrices $\diag(a \hcolon b \hcolon c)$ which preserve the form $xyz + x^3 + y^3$ are
those satisfying the equations $abc = a^3 = b^3$.
These matrices form a group isomorphic to $\mu_3$
through the parametrisation $\zeta \mapsto \diag(\zeta \hcolon \zeta^2 \hcolon 1)$,
where $\zeta^3 = 1$. We conclude that $\PGL^3_{xyz + x^3 + y^3}$ is isomorphic
to $\mu_3 \rtimes \Sigma_2$.

\bibliographystyle{plain}
\bibliography{main}

\begin{thebibliography}{10}

\bibitem{SGA7I}
{\em Groupes de monodromie en g\'eom\'etrie alg\'ebrique. {I}}.
\newblock Lecture Notes in Mathematics, Vol. 288. Springer-Verlag, Berlin,
  1972.
\newblock S{\'e}minaire de G{\'e}om{\'e}trie Alg{\'e}brique du Bois-Marie
  1967--1969 (SGA 7 I), Dirig{\'e} par A. Grothendieck. Avec la collaboration
  de M. Raynaud et D. S. Rim.

\bibitem{af2000}
Paolo Aluffi and Carel Faber.
\newblock Plane curves with small linear orbits. {I}.
\newblock {\em Ann. Inst. Fourier (Grenoble)}, 50(1):151--196, 2000.

\bibitem{behrend2007}
Kai Behrend and Ajneet Dhillon.
\newblock On the motivic class of the stack of bundles.
\newblock {\em Adv. Math.}, 212(2):617--644, 2007.

\bibitem{bergh2014}
Daniel Bergh.
\newblock The {B}inomial {T}heorem and motivic classes of universal quasi-split
  tori.
\newblock arXiv:1409.5410, 2014.

\bibitem{bf2004}
Gr{\'e}gory Berhuy and Giordano Favi.
\newblock Essential dimension of cubics.
\newblock {\em J. Algebra}, 278(1):199--216, 2004.

\bibitem{burnside1955}
W.~Burnside.
\newblock {\em Theory of groups of finite order}.
\newblock Dover Publications, Inc., New York, 1955.
\newblock 2d ed.

\bibitem{chevalley1958}
C.~Chevalley, A.~Grothendieck, and J.P. Serre.
\newblock {A}nneaux de {C}how et applications, s\'eminaire {C}. {C}hevalley.
\newblock {\em Chevalley, Ecole Normale Superieure, Paris}, 1958.

\bibitem{cfoss2008}
J.~E. Cremona, T.~A. Fisher, C.~O'Neil, D.~Simon, and M.~Stoll.
\newblock Explicit {$n$}-descent on elliptic curves. {I}. {A}lgebra.
\newblock {\em J. Reine Angew. Math.}, 615:121--155, 2008.

\bibitem{SGA3III}
Michel Demazure and A.~Grothendieck.
\newblock {\em Sch\'emas en {G}roupes ({S}\'em. {G}\'eom\'etrie
  {A}lg\'e\-brique, {I}nst. {H}autes \'{E}tudes {S}ci., 1963/64) {F}asc. 3}.
\newblock Inst. Hautes \'Etudes Sci., Paris, 1964.

\bibitem{dhillon2014}
Ajneet Dhillon and Matthew~B. Young.
\newblock The motive of the classifying stack of the orthogonal group.
\newblock arXiv:1411.2710, 2014.

\bibitem{egaII}
J.~Dieudonné and A.~Grothendieck.
\newblock \'{E}l\'ements de g\'eom\'etrie alg\'ebrique. {II}. \'{E}tude globale
  \'el\'ementaire de quelques classes de morphismes.
\newblock {\em Inst. Hautes \'Etudes Sci. Publ. Math.}, (8):222, 1961.

\bibitem{ekedahl2009ap}
Torsten Ekedahl.
\newblock Approximating classifying spaces by smooth projective varieties.
\newblock arXiv:0905.1538v1, 2008.

\bibitem{ekedahl2009gg}
Torsten Ekedahl.
\newblock The {G}rothendieck group of algebraic stacks.
\newblock arXiv:0903.3143v2, 2009.

\bibitem{ekedahl2009fg}
Torsten Ekedahl.
\newblock An invariant of a finite group.
\newblock arXiv:0903.3143v2, 2009.

\bibitem{giraud1971}
Jean Giraud.
\newblock {\em Cohomologie non ab\'elienne}.
\newblock Springer-Verlag, Berlin, 1971.
\newblock Die Grund\-lehren der mathematischen Wissenschaften, Band 179.

\bibitem{hartshorne1977}
Robin Hartshorne.
\newblock {\em Algebraic geometry}.
\newblock Springer-Verlag, New York, 1977.
\newblock Graduate Texts in Mathematics, No. 52.

\bibitem{husemoller2004}
Dale Husem{\"o}ller.
\newblock {\em Elliptic curves}, volume 111 of {\em Graduate Texts in
  Mathematics}.
\newblock Springer-Verlag, New York, second edition, 2004.
\newblock With appendices by Otto Forster, Ruth Lawrence and Stefan Theisen.

\bibitem{joyce2007}
Dominic Joyce.
\newblock Motivic invariants of {A}rtin stacks and `stack functions'.
\newblock {\em Q. J. Math.}, 58(3):345--392, 2007.

\bibitem{knutson1973}
Donald~M. Knutson.
\newblock {\em {$\lambda$}-rings and the representation theory of the symmetric
  group}, volume 308 of {\em Lecture notes in mathematics}.
\newblock Springer, 1973.

\bibitem{kraft1984}
Hanspeter Kraft.
\newblock {\em Geometrische {M}ethoden in der {I}nvariantentheorie}.
\newblock Aspects of Mathematics, D1. Friedr. Vieweg \& Sohn, Braunschweig,
  1984.

\bibitem{kresch1999}
Andrew Kresch.
\newblock Cycle groups for {A}rtin stacks.
\newblock {\em Inventiones mathematicae}, 138(3):495--536, 1999.

\bibitem{milne1980}
James~S. Milne.
\newblock {\em \'{E}tale cohomology}, volume~33 of {\em Princeton Mathematical
  Series}.
\newblock Princeton University Press, Princeton, N.J., 1980.

\bibitem{rokeaus2007}
Karl R{\"o}kaeus.
\newblock The class of a torus in the {G}rothendieck ring of varieties.
\newblock {\em Amer. J. Math.}, 133(4):939--967, 2011.

\bibitem{toen2005}
Bertrand To{\"e}n.
\newblock Grothendieck rings of {A}rtin n-stacks.
\newblock arXiv:math/0509098v3, 2005.

\bibitem{voskresenskii1973}
V.~E. Voskresenski{\u\i}.
\newblock The geometry of linear algebraic groups.
\newblock {\em Trudy Mat. Inst. Steklov.}, 132:151--161, 266, 1973.
\newblock Proc. Internat. Conf. Number Theory (Moscow, 1971).

\bibitem{voskresenskii1997}
V.~E. Voskresenski{\u\i}.
\newblock Birational geometry of linear algebraic groups and {G}alois modules.
\newblock {\em J. Math. Sci. (New York)}, 85(4):2017--2114, 1997.
\newblock Algebraic geometry, 2.

\bibitem{waterhouse1979}
William~C. Waterhouse.
\newblock {\em Introduction to affine group schemes}, volume~66 of {\em
  Graduate Texts in Mathematics}.
\newblock Springer-Verlag, New York, 1979.

\end{thebibliography}

\ifarxiv
\else
\affiliationone{
Daniel Bergh\\
Department of Mathematics,\\
Stockholm University,\\
SE-106 91 Stockholm, Sweden
\email{dbergh@gmail.com}
}
\fi
\end{document}